\documentclass[10pt]{amsart}
\usepackage{amsfonts,amsmath,amssymb,amsthm,color}
\usepackage{amsxtra}
\usepackage{color}
\usepackage{graphicx}
\usepackage{epsf}
\usepackage[font=small,format=plain,labelfont=bf,up]{caption}
\theoremstyle{plain}

\newtheorem{thm}{Theorem}[section]
\newtheorem{cor}[thm]{Corollary}
\newtheorem{lmm}[thm]{Lemma}
\newtheorem{prp}[thm]{Proposition}

\newcommand{\Mod}[1]{\mathrm{Mod}({#1})}
\newcommand{\Aut}[1]{\mathrm{Aut}({#1})}
\newcommand{\Out}[1]{\mathrm{Out}({#1})}

\newcommand{\C}[1]{\mathcal{C}({#1})}
\newcommand{\V}[1]{\mathcal{C}_{0}({#1})}
\newcommand{\G}[1]{\mathcal{C}_{1}({#1})}
\newcommand{\A}[1]{\mathcal{A}({#1})}
\newcommand{\Ov}[1]{\mathcal{O}({#1})}
\newcommand{\diam}[1]{\mathrm{diam}({#1})}

\newcommand{\grp}[1]{\langle{#1}\rangle}
\newcommand{\pdiam}[2]{\mathrm{diam}_{{#1}}({#2})}
\newcommand{\Dis}[1]{\mathrm{d}({#1})}
\newcommand{\dis}[2]{\mathrm{d}_{{#1}}({#2})}
\newcommand{\N}{\mathbb{N}}
\newcommand{\Z}{\mathbb{Z}}
\newcommand{\R}{\mathbb{R}}
\newcommand{\nada}{\emptyset}
\newcommand{\del}{\partial}
\newcommand{\sez}{\Longrightarrow}
\newcommand{\QED}{\hfill$\Box$\medskip}
\newcommand{\onto}{\twoheadrightarrow}
\newcommand{\into}{\hookrightarrow}

\newcommand{\frgt}{\backslash}
\newcommand{\tilda}[1]{\widetilde{#1}}

\title[Short-word pseudo-Anosovs]{A recipe for short-word pseudo-Anosovs}
\author{Johanna Mangahas}
\address{University of Michigan}
\email{mangahas@umich.edu}
\thanks{The author is partially supported by NSF RTG grant \#0602191}
\begin{document}
\bibliographystyle{amsalpha}
\begin{abstract}Given any generating set of any pseudo-Anosov-containing subgroup of the mapping class group of a surface, we construct a pseudo-Anosov with word length bounded by a constant depending only on the surface.  More generally, in any subgroup $G$ we find an element $f$ with the property that the minimal subsurface supporting a power of $f$ is as large as possible for elements of $G$; the same constant bounds the word length of $f$.  Along the way we find new examples of convex cocompact free subgroups of the mapping class group.
\end{abstract}
\maketitle
\section{Introduction}

Consider $\Mod{S}$, the mapping class group of a surface $S$, and its action on the isotopy classes of simple closed curves on $S$.  If all elements of a subgroup fix no common family of curves, then the group contains a \emph{pseudo-Anosov}, that is, a single element which itself fixes no finite family of curves \cite{Iv2}.  This paper answers in the affirmative Fujiwara's question of whether one can always find a ``short-word'' pseudo-Anosov (Question 3.4 in \cite{Fu1}).  Where $\Sigma$ generates the group $G$, let \emph{$\Sigma$-length} denote the length of an element of $G$ in the word metric induced by $\Sigma$.  The affirmative statement is:

\begin{thm}\label{mpa}There exists a constant $K = K(S)$ with the following property.  Suppose $G < \Mod{S}$ is finitely generated by $\Sigma$ and contains a pseudo-Anosov.  Then $G$ contains a pseudo-Anosov with $\Sigma$-length less than $K$.\end{thm}

The proof provides an explicit construction of pseudo-Anosovs from arbitrary non-pseudo-Anosov elements.  In fact, it addresses a broader question.  Roughly speaking, a pseudo-Anosov requires the whole surface for its support.  The other elements are called reducible, because they allow one to ``reduce'' the surface in question.  That is, reducible mapping classes have powers which fix proper subsurfaces and act trivially or as a pseudo-Anosov on those subsurfaces.  Mosher has described a unifying approach to both types \cite{Mo}, associating to a mapping class $f$ what he calls its active subsurface $\A{f}$.  For the sake of introduction one may think of $\A{f}$ as the smallest subsurface supporting some power of $f$ (noting that these subsurfaces and thus their inclusion are defined up to isotopy) and observe that $f$ is pseudo-Anosov exactly when $\A{f} = S$.  Then several foundational mapping class group theorems, including the Tits alternative for $\Mod{S}$ \cite{Iv2, Mc}, and subgroup structure results from \cite{BLM} and \cite{Iv1}, elegantly derive from what Mosher coined the Omnibus Subgroup Theorem\footnote{We should clarify that Mosher formulated the Omnibus Theorem to yield as corollaries certain results we need to prove Theorem~\ref{srpa}.  Mosher wanted streamlined proofs for transposing to the key of $\Out{F_n}$ (see \cite{HM} for progress).  Content to use the old results as-is, we define active subsurfaces so that the Omnibus Theorem rephrases a known theorem (see Section \ref{pure}); meanwhile we reap the benefits of its perspective and terminology.}: given a group $G < \Mod{S}$ there exists $f \in G$ such that for all $g \in G$, $\A{g} \subset \A{f}$.  Call such an $f$ \emph{full-support for} $G$.  In this paper we actually prove the following, which includes Theorem~\ref{mpa} as a special case.

\medskip
\noindent \textbf{Theorem~\ref{srpa}.} (Main Theorem) \emph{There exists a constant $K = K(S)$ such that, for any finite subset $\Sigma \subset \Mod{S}$, one may find $f$ full-support for $\grp{\Sigma}$ with $\Sigma$-length less than $K$.}\medskip

The proof of Theorem~\ref{srpa} spells out short pseudo-Anosovs explicitly, with the following core construction, concerning special pairs of pure reducible mapping classes we will call \emph{sufficiently different}.  These are pure mapping classes $a$ and $b$ with pseudo-Anosov restrictions to proper subsurfaces $A$ and $B$ respectively, such that $A$ and $B$ together fill $S$, meaning that each curve on $S$ has essential intersection with either $A$ or $B$.  The proposition also identifies subgroups whose action on the curve complex gives a quasi-isometric embedding, so that they are convex cocompact \cite{H, KL2} in the sense defined by Farb and Mosher \cite{FM} in analogy to Kleinian groups.  This last part is proven for interest, and is not necessary for the main theorem.

\begin{prp}\label{coreconstruction}There exists a constant $Q = Q(S)$  with the following property.  Suppose $a$ and $b$ are sufficiently different pure reducible mapping classes.  Then for any $n,m \geq Q$, every nontrivial element of $G = \grp{a^n, b^m}$ is pseudo-Anosov except those conjugate to powers of $a^n$ or $b^m$.  Furthermore, $G$ is a rank two free group, and all of its finitely generated all-pseudo-Anosov subgroups are convex cocompact.\end{prp}

In \cite{Ma}, the author considers a more general condition for pairs of pure reducible mapping classes, and finds $Q$ such that $G$ as above is a rank two free group, but need not contain pseudo-Anosovs (and therefore need not be convex cocompact).  A more relevant comparison is Thurston's theorem providing the first concrete examples of pseudo-Anosovs \cite{Th} (see also \cite{Pe}).  He proved that if $a$ and $b$ are Dehn twists about filling curves, then one can find an affine structure on $S$ inducing an embedding of $\grp{a,b}$ into $PSL(2,\R)$ under which hyperbolic elements of $PSL(2,\R)$ correspond to pseudo-Anosovs in $\grp{a,b}$.  In particular, every nontrivial element of the free semigroup generated by $a$ and $b^{-1}$ is pseudo-Anosov.  More recently, Hamidi-Tehrani \cite{HT} classified all subgroups generated by a pair of positive Dehn multi-twists.  In particular, if $\alpha$ and $\beta$ are multicurves whose union fills $S$, and $a$ and $b$ are compositions of positive powers of Dehn twists about components of $\alpha$ and $\beta$ respectively, then except for finitely many pairs $n,m$, $\grp{a^n,b^m}$ is a rank-two free group whose only reducible elements are those conjugate to powers of $a$ or $b$ (see also \cite{Ish}).  In a different light, one can consider Proposition~\ref{coreconstruction} a companion to a theorem of Fujiwara that generates convex cocompact free groups using bounded powers of independent pseudo-Anosovs; this theorem appears in Section \ref{schottky} as Theorem~ \ref{fuji}.

\textbf{Acknowledgments.}  Atop my lengthy debt of gratitude sit Dick Canary and Juan Souto for everything entailed in advising me to PhD (out of which work this paper emerged) and Chris Leininger for so many tools of the trade.

\section{Preliminaries}
This section consists of five parts.  The first presents basic definitions around the mapping class group and curve complex; the second relates the curve complexes of a surface and its subsurfaces.  Section \ref{coded} presents a different view of subsurfaces, which facilitates Section \ref{pure} on the case for understanding reducible mapping classes as subsurface pseudo-Anosovs.  Finally, Section \ref{CC} recalls powerful curve complex tools with which we rephrase the classification of elements of $\Mod{S}$.

\subsection{Mapping class group and curve complex}
Throughout, we consider only oriented surfaces whose genus $g$ and number of boundary components $p$ are finite.  Define the \emph{complexity} $\xi(S)$ of a surface $S$ by $\xi = 3g + p - 3$.  We neglect the case where $\xi$ is $-2$ or zero, which means $S$ is a disk, a closed torus, or a pair of pants, because these are never subsurfaces of interest, as explained in Section \ref{CCproj}.  Annuli, for which $\xi = -1$, feature throughout, but primarily as subsurfaces of higher-complexity surfaces.  Let us first assume $\xi \geq 1$, and address the annulus case after.  Note that our definitions will not distinguish between boundary components and punctures, except on an annulus.  One may find in \cite{FM} a discussion on variant definitions of the mapping class group.


Given a surface $S$, its \emph{mapping class group} $\Mod{S}$ is the discrete group of orientation-preserving homeomorphisms from $S$ to itself that setwise fix components of $\del S$.  Much (arguably, everything) about $\Mod{S}$ appears in its action on the isotopy classes of those simple closed curves on $S$ that are \emph{essential}: neither homotopically trivial nor boundary-parallel.  Let us call these classes \emph{curves} for short.  \emph{Pseudo-Anosov} mapping classes are those that fix no finite family of curves; necessarily these have infinite order.  Among the rest, we distinguish those that have finite order, and call the remaining \emph{reducible}.

The \emph{intersection} of two curves $\alpha$ and $\beta$, written $i(\alpha,\beta)$, is the minimal number of points in $\alpha' \cap \beta'$, where $\alpha'$ ranges over all representatives of the isotopy class denoted by $\alpha$ and likewise $\beta'$ ranges over representatives of $\beta$.  Often, we will use the same notation for a curve as an isotopy class and for a representative path on the surface.  For specificity and guaranteed minimal intersections, one may represent all curves by closed geodesics with respect to any pre-ordained hyperbolic metric on $S$.  Two curves \emph{fill} a surface if every curve of that surface intersects at least one of them.  \emph{Disjoint} curves have zero intersection.  All of these definitions extend to \emph{multicurves}, our term for sets of pairwise disjoint curves.

Now let us upgrade the set action of $\Mod{S}$ on curves to a simplicial action on the \emph{curve complex} of $S$, denoted $\C{S}$.  Because $\C{S}$ is a \emph{flag complex}, its data reside entirely in the one-skeleton $\G{S}$: higher-dimensional simplices appear whenever the low-dimensional simplices allow it.  Thus it suffices to consider only the graph $\G{S}$, although we usually refer to the full complex out of habit.

Curves on $S$ comprise the vertex set $\V{S}$, and an intersection rule determines the edges of $\G{S}$.  For a surface with complexity $\xi > 1$, edges join vertices representing disjoint curves.  Therefore $n$-simplices correspond to multicurves with $n$ distinct components, and $\xi$ gives the dimension of $\C{S}$.  When $\xi = 1$, $S$ is a punctured torus or four-punctured sphere.  Because on these any two distinct curves intersect, we modify the previous definition so that edges join vertices representing curves which intersect minimally for distinct curves on that surface---that is, once for the punctured torus and twice for the four-punctured sphere.  In both cases $\C{S}$ corresponds to the Farey tesellation of the upper half plane (vertices sit on rationals corresponding to slopes of curves on the torus).

Give $\G{S}$ the path metric where edges have unit length; this extends simplicially to the full complex.  For any two curves $\alpha$ and $\beta$, let $\Dis{\alpha,\beta}$ denote their distance in $\G{S}$.  Immediately one may observe that $\C{S}$ is locally infinite.  It is not obvious that $\C{S}$ is connected, but the proof is elementary \cite{Ha}.  Furthermore, it has infinite diameter \cite{MM1}.  A deep theorem of Masur and Minsky states that $\C{S}$ is $\delta$-hyperbolic, meaning that for some $\delta$, every edge of any geodesic triangle lives in the $\delta$-neighborhood of the other two edges \cite{MM1}.


Now suppose $S$ is the annulus $S^1 \times [0,1]$.  We modify the definition of the mapping class group to require that homeomorphisms and isotopy fix $\del S$ pointwise.  Parameterizing $S^1$ by angle $\theta$ and $S$ by $(\theta, t)$, the \emph{Dehn twist} on $S$ maps $(\theta, t)$ to $(\theta + 2\pi t,t)$.  The cyclic group generated by this twist is the entire mapping class group of the annulus.  In this group, let us consider all non-trivial elements pseudo-Anosov, for reasons clarified in Section \ref{CC}.  To define $\C{S}$, we contend with the fact that an annulus contains no essential curves.  Instead, each vertex of $\V{S}$ corresponds to the isotopy class of an arc connecting the two boundary components, again requiring isotopy fix boundary pointwise.  Edges connect vertices that represent arcs with disjoint interiors.  Although this complex is locally uncountable, it is not difficult to understand: the distance between two distinct vertices equals one plus the minimal number of interior points at which their representative arcs intersect.  Again, $\C{S}$ is connected, infinite diameter, and $\delta$-hyperbolic---in fact, it is quasi-isometric to the real line, and $\Mod{S}$ acts on it by translation.  Note that our ``curve complex'' for the annulus is more accurately called an \emph{arc} complex.  Generally, we aim to minimize the distinction between annular subsurfaces and subsurfaces with $\xi \geq 1$; for extended treatment of curve complexes and arc complexes, see \cite{MM2}.

\subsection{Subsurface projection}\label{CCproj}Let us relate the curve complex of $S$ to that of $S'$, where $\xi(S) \geq 1$ and $S'$ is an ``interesting'' subsurface of $S$.  Here, a subsurface is defined only up to isotopy, and assumed \emph{essential}, meaning its boundary curves are either essential in $S$ or shared with $\del S$.  This rules out the disk.  We also disregard pants, which have trivial mapping class group.  For the remainder of Section \ref{CCproj}, we assume $S'$ is connected; furthermore, either $\xi(S') \geq 1$ or $S'$ is an annular neighborhood of a curve in $S$.

For ease of exposition and the convenience of considering $\del S'$ a multicurve in $S$, let us make a convention that $\del S'$ refers only to those boundary curves essential in $S$.  In later sections we consider multiple, possibly nested subsurfaces, but these are always implicitly or explicitly contained in some largest surface $S$.

Except when $S'$ is an annulus, it is clear one can embed $\C{S'}$, or at least its vertex set, in $\C{S}$, but we seek a map in the opposite direction.  One can associate curves on the surface to curves on a subsurface via subsurface projection, a notion appearing in \cite{Iv1, Iv2}, expanded in \cite{MM2}, and recapitulated here.  In what follows, we define the projection map $\pi_{S'}$ from $\V{S}$ to the powerset $\mathcal{P}(\V{S'})$.  To start, represent $\gamma \in \V{S}$ by a curve minimally intersecting $\del S'$.

First suppose $S'$ is not an annulus.  If $i(\gamma, \del S') = 0$, then either $\gamma \subset S'$, and we let $\pi_{S'}(\gamma) = \{\gamma\}$, or $\gamma$ misses $S'$, and we let $\pi_{S'}(\gamma)$ be the empty set.  Otherwise, $\gamma$ intersects $\del S'$.  For each arc $\alpha$ of $\gamma \cap S'$, take the boundary of a regular neighborhood of $\alpha \cup \del S'$ and exclude the component curves which are not essential in $S'$.  Because $S'$ is neither pants nor annulus, some curves remain.  Let these comprise $\pi_{S'}(\gamma)$.

In the case that $S'$ is an annulus, when $i(\gamma, \del S') = 0$ we let $\pi_{S'}(\gamma)$ be the empty set.  When $\gamma$ intersects $\del S'$, one expects $\pi_{S'}(\gamma)$ to consist of the arcs of $\gamma$ intersecting $S'$.  However, ambiguity arises because curves such as $\gamma$ and $\del S'$ are defined up to $\del S$-fixing isotopy, but vertices of $\C{S'}$ represent arcs up to $\del S'$-fixing isotopy.  To remedy this, give $S$ a hyperbolic metric.  Consider the cover of $S$ corresponding to the fundamental group of $S'$ embedded in that of $S$.  This cover is a hyperbolic annulus endowed with a canonical ``boundary at infinity'' coming from the boundary of two-dimensional hyperbolic space (the unit circle, if one uses the Poincar\'{e} disk model).  Name the closed annulus $A$, and let $\C{A}$ stand in for $\C{S'}$.  Each $\gamma \in \V{S}$ has a geodesic representative which lifts to $A$, and if $\gamma$ intersects $S'$, some of the lifts connect the boundary components.  Let these comprise $\pi_S'(\gamma)$.

We will frequently say something like $\gamma$ \emph{projects to} $S'$ to mean $\pi_{S'}(\gamma)$ is not the empty set. Where $X$ is some collection of curves $\{\gamma_i\}$, let $\pi_{S'}(X) = \bigcup_i \pi_{S'}(\gamma_i)$.  If $\pi_{S'}(X)$ is not empty, let $\pdiam{S'}{X}$ denote its diameter in $\C{S'}$; omit the subscript $S'$ to mean diameter of $X$ in $\C{S}$ itself.  We are content with a map from $\C{S}$ to subsets of $\C{S'}$, rather than $\C{S'}$ directly, precisely because multicurves have bounded-diameter projection:

\begin{lmm}\label{projbound}Suppose $\gamma$ is a multicurve of $S$, and $S'$ a subsurface.  If $\gamma$ projects to $S'$, $\pdiam{S'}{\pi_{S'}(\gamma)} \leq 2$.\end{lmm}

This fact appears as Lemma 2.3 in \cite{MM2}.  Note that complexity $\xi$ in \cite{MM2} differs from our definition by three.  Also, there is a minor error in Lemma 2.3, corrected in \cite{Mi} (pg. 28), which we avoid by stating the lemma for multicurves rather than simplices of $\C{S}$---these do not coincide if $\xi(S) = 1$.

When multicurves $\alpha$ and $\beta$ both project to $S'$, define their \emph{projection distance} $\dis{S'}{\alpha,\beta}$ as $\pdiam{S'}{\pi_{S'}(\alpha) \cup \pi_{S'}(\beta)}$.  This is a ``distance'' in that it satisfies the triangle inequality and symmetry, but it need not discern curves: Lemma~\ref{projbound} implies that disjoint multicurves have a projection distance of at most two, when defined.  In the other extreme, one easily finds examples of curves close in $\C{S}$ with large projection distance in $\C{S'}$---a small illustration of the great wealth of structure that opens up when one considers not only $\C{S}$ but curve complexes of all subsurfaces (see, for example, \cite{MM2}).

\subsection{Cut-coded subsurfaces and domains}\label{coded}
We now present an alternate definition of subsurface that ducks the nuisance of disconnected subsurfaces containing annular components parallel to the boundary of other components.  These are the only subsurfaces capable of being mutually nested (via isotopy) yet not topologically equivalent.  Our technical antidote may seem tedious, but as an upside, it translates our current notion of a subsurface into an object encoded unambiguously by curves, a recurrent theme of this paper.  Moreover, the new viewpoint facilitates the next section's definition of active subsurfaces, based on Ivanov's work on $\Mod{S}$ subgroups.  The efficient reader is welcome to skim the definition, taking note of Lemmas~\ref{threecases} and \ref{nestseq}, and rely on Figure \ref{figcutcoded} for intuition.

A \emph{cut-coded subsurface} of $S$ consists of two pieces of information: (1) a multicurve $\gamma$ and (2) a partition of the non-pants components of $S \frgt \gamma$ into two sets: excluded and included components.

Let us clarify what subsurface $S'$ this data is meant to describe.  Label the component curves of $\gamma$ by $\gamma_1$, $\dots \gamma_n$, and the included components $A_1, \dots A_m$.  The multicurve $\gamma$ contains each component of $\del A_i$, but some $\gamma_j$ may not be a boundary curve for any $A_i$.  Let $N_j$ be regular neighborhoods of those $\gamma_j$ belonging to no $\del A_i$.  Then $S'$ consists of the union of the $A_i$, seen as subsurfaces, and the implied annuli $N_j$.  Ignoring multiplicity, $\del S'$ and $\gamma$ are the same multicurve.

\begin{figure}[h]
\begin{center}
\includegraphics[width=5in]{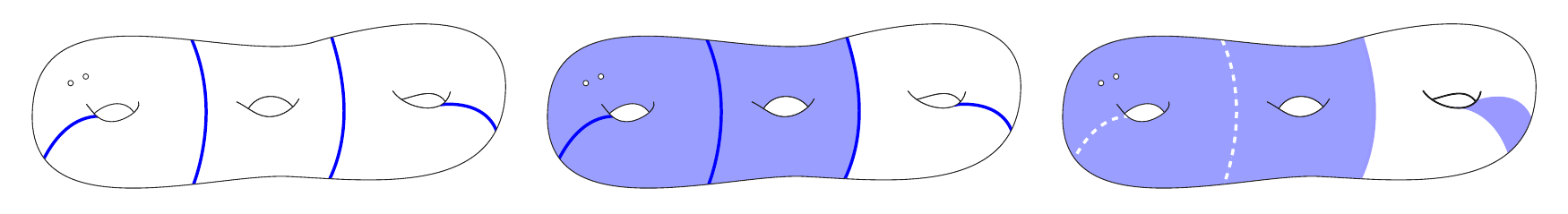}
\caption{\emph{How to define a cut-coded subsurface.}  From left to right: (1) Set multicurve.  (2) Choose included components (shaded).  (3) Domains (shaded) correspond to included components and implied annuli.}
\label{figcutcoded}
\end{center}
\end{figure}

Call the $A_i$ \emph{component domains} and the $N_j$ \emph{annular domains} of $S'$.  We use \emph{domain} to refer to either kind, or any subsurface that may appear as a domain---in other words, any connected (essential) non-pants subsurface.

Not all subsurfaces are cut-coded, as the latter never include pants, parallel annuli, or annuli parallel to the boundary of a component domain.  Cut-coded subsurfaces are exactly those appearing as active subsurfaces of mapping classes, which the next section details.

The cut-coded subsurfaces of $S$ admit a partial order $\subset$ detected by subsurface projection.  Say a subsurface $A$ \emph{nests} in $B$ if $A$ may be isotoped into $B$.  Now suppose $A$ and $B$ are cut-coded subsurfaces with domains $A_i$ and $B_j$ respectively.  Say $A \subset B$ if every $A_i$ nests in some $B_j$.  Transitivity and reflexivity of $\subset$ are obvious, but antisymmetry is relatively special.  If $A \subset B$ and $B \subset A$, one can check $A$ and $B$ are given by the same data as cut-coded subsurfaces, ultimately because they never contain an annulus parallel to another domain.  In contrast, general disconnected subsurfaces can be mutually nested but not isotopic.

Call two subsurfaces \emph{disjoint} if they may be isotoped apart, and \emph{overlapping} if they are neither disjoint nor nested; note that overlapping subsurfaces are distinct by definition.  Projection determines relations between domains:

\begin{lmm}\label{threecases}Suppose $A$ and $B$ are domains in $S$.
\begin{itemize}
\item[(i)] $\pi_{A}(\del B)$ is empty if and only if $A$ nests in $B$ or its complement.
\item[(ii)] $A$ and $B$ overlap if and only if $\del A$ projects to $B$ and $\del B$ to $A$.
\item[(iii)] If $\del A$ projects to $B$ but $\pi_{A}(\del B)$ is empty, $A$ nests in $B$.
\item[(iv)] If $\pi_{A}(\del B)$ and $\pi_{B}(\del A)$ are both empty, $A$ and $B$ are equal or disjoint.
\end{itemize}
\end{lmm}

\begin{proof}  Statements (ii) -- (iv) derive from (i).  The forward implication of (i) requires connectedness of $A$: one may choose curves $\alpha_1$ and $\alpha_2$ that fill $A$, so that $A$ is represented by a regular neighborhood of $\alpha_1 \cup \alpha_2$ with disks and annuli added to fill in homotopically trivial or $\del S$-parallel boundary components.  If $\pi_{A}(\del B)$ is empty, then $\del B$ is disjoint from $\alpha_1 \cup \alpha_2$.  Therefore the two curves, and consequently $A$ itself, can be isotoped either entirely inside $B$ or into its complement.  The reverse implication is self-evident.\end{proof}

Consider a strictly increasing sequence $A^1 \subsetneq A^2 \subsetneq \dots \subsetneq A^m$ of cut-coded subsurfaces of $S$.  Let $\alpha$ be a maximal multicurve (i.e., a pants decomposition) including $\del A_i$ for all $i$.  Each step of the sequence corresponds to some curve of $\alpha$ appearing for the first time as a component of $\del A_k$ or an essential curve in $A_k$, so the maximum length of the sequence is twice the number of components of $\alpha$.  We have just observed:

\begin{lmm}\label{nestseq}If $\xi(S) \geq 1$, a strictly increasing sequence of nonempty cut-coded subsurfaces of $S$ has length at most $2\xi$.\end{lmm}

It is easy to construct sequences realizing the upper bound.

\subsection{Active subsurfaces and the Omnibus Theorem}\label{pure}
Thurston originally classified elements of $\Mod{S}$ by their action on the space of projective measured foliations, a piecewise linear space obtained by completing and projectivizing the space of weighted curves on $S$.  He defined a \emph{pseudo-Anosov diffeomorphism} as one that fixes a pair of transverse projective measured foliations, and proved that any pseudo-Anosov mapping class (i.e., any mapping class fixing no multicurve) has a representative pseudo-Anosov diffeomorphism.  In fact he proved more:

\begin{thm}[Thurston, \cite{Th},\cite{FLP}]\label{ThurpA}Every element $f$ of $\Mod{S}$ has a representative diffeomorphism $F$ such that, after cutting $S$ along some 1-dimensional submanifold $C$, $F$ restricts to a pseudo-Anosov diffeomorphism on the union of some components of $S \frgt C$, and has finite-order on the union of the rest.\end{thm}

Birman-Lubotzky-McCarthy proved that $C$ represents a unique isotopy class, and they gave a simple way to find it \cite{BLM}.  We call this isotopy class the \emph{canonical reduction multicurve} of $f$.  Ivanov generalized to subgroups both the classification and the canonical cutting method \cite{Iv2}.  We build the notion of an \emph{active subsurface} according to this last, largest perspective.  Let us emphasize that, while we aim to paint a picture that seems perfectly natural, the validity of the definitions in this section depends on multiple lemmas and theorems from \cite{Iv2}, which serves as reference for assertions presented without proof.

Ivanov identified a congenial property of many mapping classes.  Call a mapping class \emph{pure} if, in the theorem above, it restricts to either the identity or a pseudo-Anosov on each component (in particular, no components are permuted).  Because one may take the empty set for $C$ in Theorem~\ref{ThurpA}, any pseudo-Anosov mapping classes is pure.  The property of being pure is most useful for reducible mapping classes, because cutting along $C$ ``reduces'' $S$ to a collection of smaller subsurfaces.  The point in what follows is to formalize this procedure.

Call a subgroup pure if it consists entirely of pure mapping classes.  Nontrivial pure mapping classes have infinite order, so they let us ignore the complications of torsion.  Fortunately, the mapping class group has finite-index pure subgroups (Theorem 3, \cite{Iv2}).  In particular, one can take the kernel of the homomorphism $\Mod{S} \onto \Aut{H_1(S,Z/3\Z)}$, induced by the action of $\Mod{S}$ on homology.  Name this subgroup $\Gamma(S)$.

Let us first define the canonical reduction multicurve $\sigma(G)$ and active subsurface $\A{G}$ of a pure subgroup $G < \Mod{S}$.  The multicurve $\sigma(G)$ consists of all curves $\gamma$ such that (i) $G$ fixes $\gamma$, and (ii) if some curve $\beta$ intersects $\gamma$, then $G$ does not fix $\beta$.  The active subsurface $\A{G}$ is cut-coded with multicurve $\gamma$.  Included components correspond to those on which some element of $G$ ``acts pseudo-Anosov.''  Specifically, for each component $Q$ of $S \frgt \sigma(G)$, we have a homomorphism $\rho_Q: G \to \Mod{Q}$ such that $\rho_Q(g)$ is the mapping class of $F|_{Q}$, where $F$ is a homeomorphism representing $g$.  Let $G_Q$ denote the image $\rho_Q(G)$.  Because $G$ is pure, for each component $Q$, the image $G_Q$ either contains a pseudo-Anosov, in which case $Q$ is included, or is the trivial group, in which case $Q$ is excluded (Theorem 7.16, \cite{Iv2}).

It follows that the annular domains of $\A{G}$ correspond to neighborhoods of those $\gamma \in \sigma(G)$ that only bound components $Q$ for which $G_Q$ is trivial.  Because by definition $\gamma$ is not superfluous, some $g \in G$ restricts, in a neighborhood of $\gamma$, to a power of a Dehn twist about $\gamma$.  Section \ref{CC} justifies why we consider Dehn twists annulus pseudo-Anosovs.

If $G$ is an arbitrary subgroup, choose a finite-index pure subgroup $G'$ and define $\sigma(G) = \sigma(G')$ and $\A{G}= \A{G'}$.  Any choice of $G'$ gives the same multicurve, and one can always take $G' = G \cap \Gamma(S)$.  $G$ acts on $S \frgt \sigma(G)$, although its elements may permute the components.  For each component $Q$, one can define $\rho_Q$ on the finite-index subgroup of $G$ stabilizing $Q$.  Each image $G_Q$ is finite or contains a pseudo-Anosov, and $\sigma(G)$ is the minimal multicurve with this property (Theorem 7.16, \cite{Iv2}).  It follows that an infinite subgroup $G$ contains a pseudo-Anosov if and only if $\sigma(G)$ is empty; let us call such a subgroup \emph{irreducible}.

For any mapping class $g$, let $\sigma(g) = \sigma(\grp{g})$ and $\A{g} = \A{\grp{g}}$.  In this terminology, we recast Ivanov's Theorem 6.3 \cite{Iv2} as follows:
\begin{thm}\label{omnibus}For any $G < \Mod{S}$ there exists $f \in G$ such that $\A{f} = \A{G}$.\end{thm}

After Lemma~\ref{nest1} below, we can recognize Theorem~\ref{omnibus} as Mosher's Omnibus Subgroup Theorem.  We do not require this theorem for our proofs, only the validity of the definitions on which it is based.  In fact, one can prove Theorem~\ref{mpa} with no mention of active subsurfaces for non-pure subgroups.  However, its definition allows us to state the Main Theorem, Theorem~\ref{srpa}, which takes the more general, perhaps more useful point of view of Theorem~\ref{omnibus}, adding the benefit of an $f$ with bounded word length.

Facts about active subsurfaces occupy the remainder of this section.  Directly from definitions, we derive Lemma~\ref{fixmove} below.  With this we obtain Lemmas~\ref{nest1} -- \ref{nest}, which enable our proof of the main theorem.  Let us refer to domains of $\A{G}$ or $\A{g}$ as domains of $G$ or $g$ respectively.  Say $G$ \emph{moves} the curve $\gamma \in \V{S}$ if some element of $G$ does not fix $\gamma$.

\begin{lmm}\label{fixmove}Suppose $H$ and $G$ are pure subgroups of $\Mod{S}$ and $\gamma \in \V{S}$.
\begin{itemize}
\item[(a)] $H$ moves $\gamma$ if and only if $\gamma$ projects to some domain of $H$.
\item[(b)] $H$ and $G$ fix and move the same curves if and only if $\A{H} = \A{G}$.
\end{itemize}
\end{lmm}

\begin{lmm}\label{nest1}If $H < G$ then $\A{H} \subset \A{G}$.\end{lmm}

\begin{proof}We may replace $H$ and $G$ by the pure subgroups obtained by intersecting with $\Gamma(S)$.  Suppose $A$ is a domain of $G$ and $B$ a domain of $H$.  Because $G$, and thus $H$, fixes $\del A$, $\pi_{B}(\del A)$ is empty.  Lemma~\ref{threecases} guarantees that $B$ nests in $A$ or its complement.  Now suppose $B$ is in the complement of every domain of $G$.  Then any curve essential in $B$ is fixed by $G$, thus by $H$.  This contradicts that $B$ is a domain of $H$, unless $B$ has no essential curves.  Thus $B$ is an annulus and $\del B$ consists of (two copies of) a single curve $\beta$.  Any $\gamma$ intersecting $\beta$ is moved by $H$, hence by $G$.  But because $G$ fixes $\gamma$, this means $\gamma \in \sigma(G)$.  Because $B$ is the annulus around $\gamma$, $B$ isotopes into $\A{G}$.  Thus every domain of $H$ nests in $\A{G}$, which means $\A{H} \subset \A{G}$.\end{proof}

\begin{figure}[h]\label{figactive}
\begin{center}
\includegraphics[width=5in]{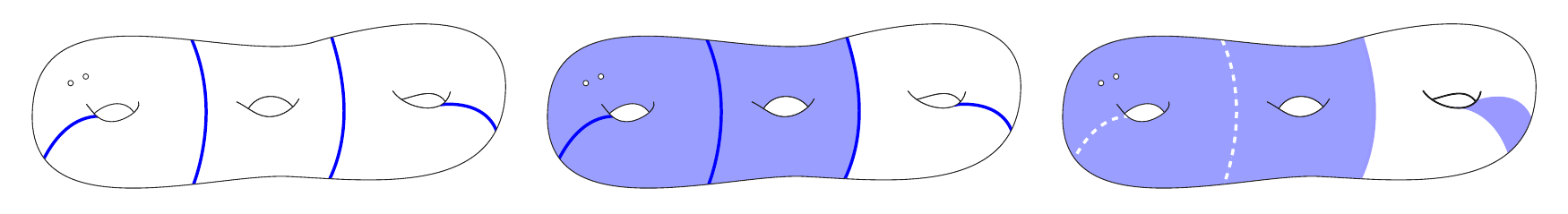}
\caption{\emph{How to define the active subsurface of a group $G$.}  (1) Cut along canonical reduction multicurve.  (2) Include components (shaded) on which $G$ induces an irreducible subgroup.  (3) Every element of $G$ has a power supported on included components and implied annuli, by Lemma~\ref{nest1}.  By Theorem~\ref{omnibus}, some element has exactly this support.}
\end{center}
\end{figure}

In general, let $\A{X_1,X_2, \dots}$ denote the active subsurface of the group generated by $X_1,X_2, \dots$, where $X_i$ may be either elements or subgroups of a mapping class group.

\begin{lmm}\label{nest2}Suppose $H$ and $G$ are pure elements or subgroups of $\Mod{S}$.  Then $\A{H} \subset \A{G}$ if and only if $\A{G} = \A{G,H}$.\end{lmm}

\begin{proof}Observe that $\A{H} \subset \A{G}$ implies that any curve fixed by $G$ is fixed by $H$, and hence $G$ and $\grp{G,H}$ fix the same curves.  Apply Lemma~\ref{fixmove} to $G < \grp{G,H}$ for the forward direction, and Lemma~\ref{nest1} to $H < \grp{G,H}$ for the reverse.
\end{proof}

\begin{lmm}\label{nest} Let $G$ be a pure subgroup of $\Mod{S}$ generated by $\{H, H_1, \dots H_k, \dots\}$, where $H_k$ are groups or elements.  Either $\A{H} = \A{G}$ or $\A{H} \neq \A{H,H_k}$ for some $k$.\end{lmm}

\begin{proof}Suppose $\A{H} = \A{H,H_k}$ for all $k$.  Let $K_i = \grp{H, H_1, \dots H_i}$.  Using induction and Lemma~\ref{nest2} one can show $\A{H} = \A{K_i}$ for all $i$.  Because the $K_i$ exhaust $G$, Lemma~\ref{nest1} implies $\A{g} \subset \A{H}$ for all $g \in G$.  By Lemma~\ref{fixmove} one knows any curve fixed by $H$ has empty projection to all $\A{g}$, so $H$ and $G$ fix and move the same curves, and consequently $\A{H} = \A{G}$.\end{proof}

\subsection{Machinations in the curve complex}\label{CC}
This section collects several important curve complex results that provide the foundation for our proofs.  Together, these results link mapping class behavior with curve complex geometry.  Following our theme of interpreting $\Mod{S}$ via $\C{S}$, we note that these results also lead to a $\C{S}$-centric $\Mod{S}$ classification.

By definition, reducible mapping classes and finite-order mapping classes have bounded orbits in $\C{S}$.  That pseudo-Anosovs have infinite-diameter orbits is corollary to a theorem of Masur and Minsky:

\begin{thm}[Minimal translation of pseudo-Anosovs \cite{MM1}] \label{trans} There exists $c = c(S) > 0$ such that, for any pseudo-Anosov $g \in \Mod{S}$, vertex $\gamma \in \mathcal{C}_0(S)$, and nonzero integer $n$, \[d_{S}(g^n(\gamma),\gamma) \geq c|n| .\] \end{thm}

This gives us a way to recognize pseudo-Anosovs.  It also suggests an alternate classification of elements on $\Mod{S}$, by whether they have finite, finite-diameter, or infinite-diameter orbits in $\C{S}$.  If one takes that classification as a starting point, Dehn twists clearly qualify as pseudo-Anosovs for the annulus mapping class group.

For pure mapping classes, one may even refine this classification.  Suppose $g \in \Mod{S}$ is pure and has a domain $Y$.  The observation that, for all $\gamma \in \V{S}$, $\pi_Y(g(\gamma)) = g(\pi_Y(\gamma))$, yields:

\begin{cor}[Minimal translation on subsurfaces] \label{strans} There exists $c = c(S) > 0$ such that, for any pure element $g \in \Mod{S}$ with domain $Y$, vertex $\gamma \in \mathcal{C}_0(S)$ such that $\pi_Y(\gamma) \neq \nada$, and nonzero integer $n$, \[d_{Y}(g^n(\gamma),\gamma) \geq c|n| .\] \end{cor}

Thus every nontrivial pure element $g$ has infinite-diameter orbits in the curve complexes of its domains.  Moreover, an orbit of $g$ in the curve complex of one of its domains projects to a bounded diameter subset of the curve complex of any subsurface properly nested in that domain.  This is because, if some $g$-orbit projects to a domain, most of its curves approximate the limiting laminations of $g$ which fill that domain.  Projection to the properly nested subsurface will not distinguish between the approximating curves, so the orbit looks roughly constant at large-magnitude powers of $g$.  This is the heuristic behind another theorem of Masur and Minsky, which along with Corollary \ref{strans} is crucial to our core construction in Section \ref{blueprint}:

\begin{thm}[Bounded geodesic image \cite{MM2}]\label{bgi} Let $Y$ be a proper domain of $S$.  Let $\mathcal{G}$ be a geodesic in $\C{S}$ whose vertices each project to $Y$.  Then there is a constant $M = M(S)$ such that \[\pdiam{Y}{\mathcal{G}} \leq M.\] \end{thm}

This theorem enabled Behrstock to obtain Lemma~\ref{behrstock} below, which Chapter \ref{recipe} frequently employs.  Here we give the elementary proof, with constructive constants, by Chris Leininger.  Most of it appeared previously in \cite{Ma}; this version adds the possibility of annular domains.

\begin{lmm} [Behrstock \cite{Be}] \label{behrstock}For any pair of overlapping domains $Y$ and $Z$ and any multicurve $x$ projecting to both,
\[ \dis{Y}{x,\del Z} \geq 10 \sez \dis{Z}{x,\del Y} \leq 4 \]
\end{lmm}

\begin{proof} First we gather the facts that prove the lemma when neither $Y$ nor $Z$ is an annulus.  Suppose $S'$ is a subsurface of $S$ and $\xi(S),\xi(S') \geq 1$.  Let $u_0$ and $v_0$ be curves on $S$ which minimally intersect $S'$ in sets of arcs.  Suppose $a_u$ is one these arcs for $u_0$, and $u$ a component of the boundary of a neighborhood of $a_u \cup \del S'$; suppose $a_v$ and $v$ play the same role for $v_0$.  Then $u \in \pi_{S'}(u_0)$ and $v \in \pi_{S'}(v_0)$.  Define intersection number of arcs to be minimal over isotopy fixing the boundary setwise but not necessarily pointwise.  One has:

\begin{itemize}
\item[(1)] If $i(a_u,a_v) = 0$, then $\dis{S'}{u_0,v_0} \leq 4$
\item[(2)] If $i(u,v) > 0$, then $i(u,v) \geq 2^{(\dis{S'}{u,v} -2)/2}$
\item[(3)] $i(u,v) \leq 2 + 4 \cdot i(a_u,a_v)$
\end{itemize}

Statement (1) follows from the proof of Lemma~\ref{projbound} (Lemma 2.3 in \cite{MM2}).  Straightforward induction proves (2), which Hempel records as Lemma 2.1 in \cite{He}.  Fact (3) is the observation that essential curves from the regular neighborhoods of $a_u \cup \del S'$ and $a_v \cup \del S'$ intersect at most four times near every intersection of $a_u$ and $a_v$, plus at most two more times near $\del S'$.

Now assume $\xi(Y),\xi(Z) \geq 1$.  Because $\dis{Y}{x,\del Z} \geq 10 > 2$, the diameter is realized by curves $u \in \pi_{Y}(x), v \in \pi_{Y}(\del Z)$ such that, by (2), $i(u,v) \geq 2^4 = 16$.  By the definition of $\pi_Y$, these $u$ and $v$ come from arcs $a_u \subset x \cap Y$ and $a_v \subset \del Z \cap Y$ respectively.  By (3), $i(a_u,a_v) \geq (16-2)/4 > 3$.  Thus $a_u$ is an arc of $x$ intersected thrice by an arc $a_v$ of $\del Z$, within the subsurface $Y$.  Observe that one of the segments of $a_u$ between points of intersection must lie within $Z$.  This segment is an arc $a_x$ of $x$ disjoint from arcs of $\del Y$ in $Z$.  Fact (1) implies $\dis{Z}{x,\del Y} \leq 4$.

\begin{figure}[h]
\begin{center}
\includegraphics{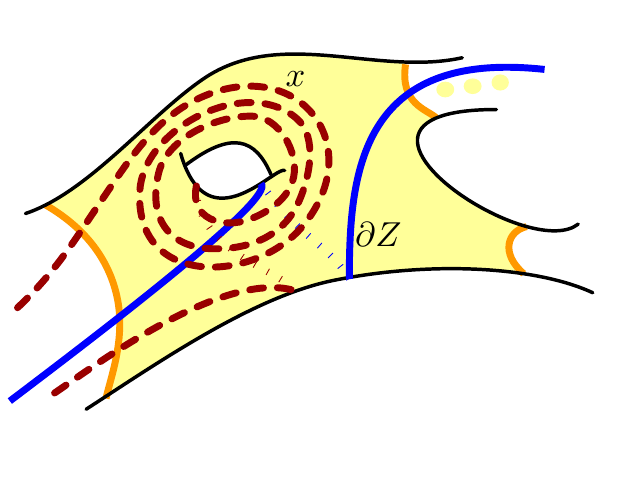}
\caption{\emph{Key point of Behrstock's lemma.}  Because $x$ and $\del Z$ have large projection distance in $Y$, one can find an arc of $x$ intersecting an arc of $\del Z$ three times.  Extra-fine dashed lines represent where curves run behind the surface.}
\end{center}
\end{figure}

The main idea of this proof works for annular domains after a few more relevant facts.   Endow $S$ with a hyperbolic metric and let $A$ be an embedded annulus with geodesic core curve $\alpha$; let $\tilda{A}$ be the corresponding annular cover of $S$.  Let $u$ and $v$ be geodesic curves in $S$ with lifts $\tilda{u}$ and $\tilda{v}$ traversing the core curve of $\tilda{A}$.  We have already mentioned
\begin{itemize}
\item[(4)] $\dis{A}{u,v} = i(\tilda{u},\tilda{v}) + 1$.
\end{itemize}
Let $\tilda{\alpha}$ be the unique lift of $\alpha$ corresponding to the core curve of $\tilda{A}$, and let $\tilda{\alpha}_1$ and $\tilda{\alpha}_2$ be the first lifts of $\alpha$ intersecting $\tilda{u}$ on each side.  On the sides of $\tilda{\alpha}_i$ opposite $\tilda{\alpha}$, $\tilda{A}$ is isometric to its pre-image in the universal cover $\tilda{S} = \mathbb{H}^2$.  Because geodesics intersect only once in $\mathbb{H}^2$,
\begin{itemize}
\item[(5)] at most two of the intersections of $\tilda{u}$ and $\tilda{v}$ occur outside the open segment of $\tilda{u}$ between the $\tilda{\alpha}_i$.
\end{itemize}

Now we can retrace the proof above, augmenting it to address the possibility that $Y$ or $Z$ is an annulus.  Under some hyperbolic metric, $x, \del Y,$ and $\del Z$ have geodesic representatives.  Suppose $Y$ is an annulus with geodesic core curve $y$ and $\tilda{Y}$ is the corresponding annular cover of $S$.  If $\dis{Y}{x,\del Z} \geq 10$, (4) implies that some lifts $\tilda{x}$ and $\tilda{\del Z}$ in $\tilda{Y}$ intersect at least nine times.  By (5), at least three of these intersections occur on an open segment of $\tilda{x}$ between consecutive lifts of $y$, and a neighborhood of this segment embeds in $S$.  As before, one finds an arc of $x$ intersecting $\del Z$ at least three times in a neighborhood disjoint from $\del Y$.  If $Z$ is not an annulus, one repeats the conclusion of the first argument.

\begin{figure}[h]
\begin{center}
\includegraphics{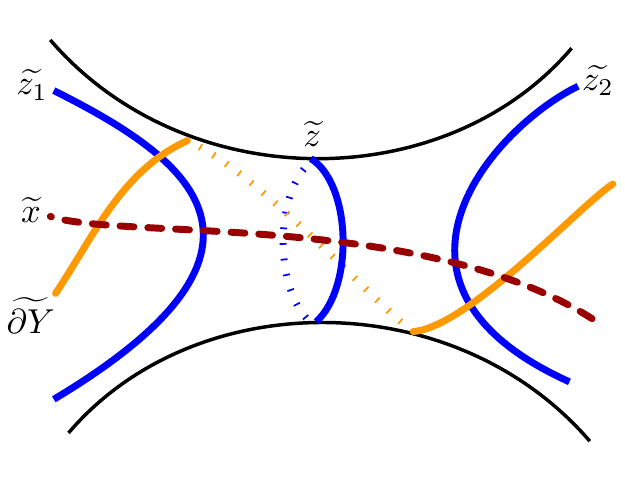}
\caption{\emph{Behrstock's lemma for the annulus.}  If $x$ and $z$ intersect three times with no intersections of $\del Y$ in between, then in the annular cover corresponding to $z$, $\tilda{x}$ can intersect $\tilda{\del Y}$ no more than twice.  Extra-fine dashed lines represent parts of curves on the far side of the annulus.}
\end{center}
\end{figure}

On the other hand, suppose that $Z$ is an annulus with geodesic core curve $z$, and $Y$ is any domain.  The arguments thus far tell us one has an arc of $x$ thrice intersecting $z$ in a neighborhood of that arc disjoint from $\del Y$.  In the annular cover corresponding to $Z$, the three intersections correspond to a lift $\tilda{x}$ of $x$ intersecting the closed lift $\tilda{z}$ of $z$ and adjacent lifts $\tilda{z}_1$, $\tilda{z}_2$ on each side.  Any lift of $\del Y$ cannot intersect $\tilda{x}$ between these intersections, so (5) implies $i(\tilda{x},\tilda{\del Y}) \leq 2$.  Fact (4) implies $\dis{Z}{x,\del y} \leq 3$.
\end{proof}

\noindent \textbf{Remark.}  Behrstock's lemma implies that $\C{S}$-orbits of a mapping class $g$ have bounded projection to subsurfaces that overlap with domains of $g$.  Using stronger results, Theorem~\ref{bgi} in particular, one can prove that the $g$-orbit of any $\gamma \in \C{S}$ projects to an unbounded set in the curve complex of $S'$ if and only if the orbit projects to $S'$ and $S'$ is a domain of $g$.  Thus one obtains a refined, curve-complex-based classification of elements of $\Mod{S}$, by letting the phrase ``$g$ is pseudo-Anosov on $S'$'' mean that the $g$-orbit of some $\gamma \in \C{S}$ has infinite-diameter projection to $\C{S'}$.

\section{The recipe}\label{recipe}
We begin at the end, proving the Main Theorem, modulo two propositions, in Section \ref{prf}.  There, we reduce the proof to the problem of generating a short-word full-support mapping class when the generating set has only two pure elements; even then, we must simultaneously accommodate multiple scenarios happening on different, non-interacting subsurfaces.  The actual construction of pseudo-Anosovs is left to Section \ref{blueprint}, where in fact we identify all pseudo-Anosov elements in any group generated by two ``sufficiently different'' pure reducible mapping classes.  That is, we prove Proposition~\ref{coreconstruction} of the introduction, which gives more than is needed for the main theorems.  In general one is not so lucky as to start out with the condition on generators required for Proposition~\ref{coreconstruction} to apply.  Thus we also need a recipe for writing a sufficiently different pair of mapping classes, given an arbitrary pair of pure reducible mapping classes generating an irreducible subgroup.  Proposition~\ref{conj} fills that need; its proof occupies Section \ref{gencon}.

\subsection{Writing the short word}\label{prf}

Restated in the terminology introduced in Section \ref{pure}, our goal is to prove the following:

\begin{thm}[Main Theorem]\label{srpa}There exists a constant $K = K(S)$ with the property that, for any subset $\Sigma \subset \Mod{S}$, there exists $f \in \grp{\Sigma}$ with $\Sigma$-length less than $K$, such that $\A{f} = \A{\Sigma}$.\end{thm}

First, let us narrow our starting point.  By definition, the finite-index pure subgroup $H = \grp{\Sigma} \cap \Gamma(S)$ has the same active subsurface as $\grp{\Sigma}$.  Suppose the index of $H$ in $\grp{\Sigma}$ is $d$.  Lemma 3.4 of Shalen and Wagreich \cite{SW} provides a generating set for $H$ of words less than $2d-1$ in length according to the original generating set.  Although they state the lemma for finite generating sets, nothing prevents the proof from applying to a general group.  Therefore, if we find a full-support mapping class in $H$ whose word length is less than $l_1$ in the new generating set, its word length is less than $l_1(2d-1)$ in the original generating set.  Recall that $\Gamma(S)$ is the kernel of the $\Mod{S}$ action on homology with coefficients in $\Z/3\Z$; its index, $|\Aut{H_1(S,\Z/3\Z)}|$, gives an upper bound for $d$.

Let $\Sigma' = \{h_i\}$ be the new generating set for $H$.  Renumbering as necessary, Lemmas~\ref{nestseq} and \ref{nest} provide a sequence $\A{h_1} \subsetneq \A{h_1,h_2} \subsetneq \A{h_1,h_2,h_3} \dots$ terminating in at most $2\xi(S)$ steps at $\A{h_1,h_2,\dots,h_k}= \A{G}$.  Let $p_1 = h_1$ and suppose for each subgroup $\grp{p_{i-1},h_i}$ we can spell a full-support element $p_i$ (i.e. $\A{p_i} = \A{p_{i-1},h_i}$) with word length less than $l_2$ in the generating set $\{p_{i-1},h_i\}$.  By Lemma~\ref{nest1}, $\A{h_i} \subset \A{p_i}$; inductively assuming that $\A{p_{i-1}} = \A{h_1,\dots,h_{i-1}}$, we also know $\A{h_j} \subset \A{p_{i-1}} \subset \A{p_i}$ for all $j < i$.  Then Lemmas~\ref{nest1} -- \ref{nest} tell us $\A{p_i} = \A{h_1,\dots,h_i}$, for all $i$.  In particular, $p_k$ is full-support for $H$ and has $\Sigma'$-length less than $l_2^k$, where $k \leq 2\xi$.

We have reduced Theorem~\ref{srpa} to the case where $A$ consists of a pair of pure mapping classes $a$ and $b$.  Let $G = \grp{a,b}$.  Finding an element of full-support for $G$ is equivalent to finding a pseudo-Anosov on $\A{G}$, although we contend with the fact that $\A{G}$ may well be disconnected.  On every domain of $\A{G}$, one of the following possibilities occurs:
\begin{itemize}
\item[(i)] $a$ and $b$ are pseudo-Anosov.
\item[(ii)] $a$ and $b$ are both reducible.
\item[(iii)] One of $a$ and $b$ is pseudo-Anosov and the other is reducible.
\end{itemize}

To handle the first possibility, we quote a theorem of Fujiwara.  An earlier incarnation of this theorem inspired the short-word question in the first place.  Call a pair of pseudo-Anosovs \emph{independent} if all pairs of nontrivial powers fail to commute.  In torsion-free groups such as those we consider, either two pseudo-Anosovs generate a cyclic subgroup, or they are independent (see the proof of Theorem 5.12 in \cite{Iv2}).  In the latter case, Fujiwara's theorem applies:

\begin{thm}[Fujiwara \cite{Fu2}, partial version]\label{fuji}There exists a constant $L = L(S)$ with the following property.  Suppose $a,b \in \Mod{S}$ are independent pseudo-Anosovs.  Then for any $n,m \geq L$, $\grp{a^n,b^m}$ is an all-pseudo-Anosov rank-two free group.\end{thm}

Case (ii) relies on two results.  The first, Proposition~\ref{zigzag} below, is a subset of Proposition~\ref{coreconstruction} from the introduction; both are proved in Section \ref{blueprint}.  Define $Q = Q(S) = \max\{3,(2M + 4)/c\}$ where $M=M(S)$ is the constant from Theorem~\ref{bgi} and $c=c(S)$ is the minimal translation length from Corollary \ref{strans}.  Recall that a pair of pure reducible mapping classes $a$ and $b$ are \emph{sufficiently different} if some $A \in \A{a}$ and $B \in \A{b}$ together fill $S$.

\begin{prp}\label{zigzag} For any $m,n > Q$ and sufficiently different pure mapping classes $a$ and $b$, the group $\grp{a^m, b^n}$ is a rank-two free group and its elements are each either pseudo-Anosov or conjugate to a power of $a$ or $b$.\end{prp}

Proposition \ref{zigzag} would be irrelevant if not for Proposition~\ref{conj}, whose proof we leave to Section \ref{gencon}.  Let $\xi = \xi(S)$ and $c=c(S)$ as above.  The following proposition works for any domain $S' \subset S$.

\begin{prp}\label{conj}  Suppose $a$ and $b$ are pure reducible mapping classes in $\Mod{S'}$ and $\grp{a,b}$ is irreducible.  For any $n \geq \xi - 1$ and $k \geq 20/c$, the words $$a_1 = (b^ka^k)^nb^k \cdot a \cdot ((b^ka^k)^nb^k)^{-1} \quad \text{and} \quad b_1 = (a^kb^k)^na^k \cdot b \cdot ((a^kb^k)^na^k)^{-1}$$ are sufficiently different pure reducible mapping classes.
\end{prp}

We handle case (iii) by converting it to either case (i) or (ii), so with the three results above we may finish the proof.  Recall we face the situation where $a$ and $b$ are pure reducible mapping classes generating a group $G$ with possibly disconnected active subsurface.  Our task is to write a word in $a$ and $b$ that induces a pseudo-Anosov on every domain of $G$.

Let $a_1$ and $b_1$ be as in Proposition~\ref{conj}.  Because these are simply conjugates of $a$ and $b$, they fulfill the same possibilities (i)-(iii) on each domain of $\A{G}$.  Let $\bar{L} = \max\{L(S'): S' \in T\}$, where $T$ is the finite set of topological types of domains of $S$ (i.e., connected non-pants subsurfaces), and $L$ is the constant in Theorem~\ref{fuji}.  Similarly let $\bar{Q} = \max\{Q(S'): S' \in T\}$.  Choose $P \geq \max\{\bar{L},\bar{Q}\}$.  Consider the following word: $$w = b_1^Pa_1^Pb_1^{-P}a_1^{P}$$  On domains where possibility (i) holds, either $a_1$ and $b_1$ commute and $w = a_1^{2P}$, a pseudo-Anosov, or $a_1$ and $b_1$ are independent and $w$ is pseudo-Anosov by Theorem~\ref{fuji}.  On domains where possibility (ii) holds, $w$ is pseudo-Anosov by Propositions~\ref{zigzag} and \ref{conj}.  On domains where possibility (iii) holds, and $a_1$ is the pseudo-Anosov, we see by re-writing $w$ as $b_1^Pa_1^Pb_1^{-P} \cdot a_1^{P}$ that it is the product of powers of pseudo-Anosovs, so we may proceed as we did for case (i).  Otherwise $a_1$ is reducible, and for any $\gamma \in \sigma(a_1)$, $\dis{X}{b_1^P(\gamma),\gamma} \geq Pc \geq Qc \geq 3$, where the leftmost inequality employs Corollary \ref{strans} and the rest are by construction.  In particular, $\gamma \in \sigma(a_1)$ and $b_1^P(\gamma) \in \sigma(b_1^Pa_1b_1^{-P})$ together fill the domain, so $a_1$ and $b_1^Pa_1b_1^{-P}$ are sufficiently different pure reducible mapping classes.  Then Proposition~\ref{zigzag} guarantees $w$ is pseudo-Anosov.  Thus on all domains of $\A{G}$, $w$ is pseudo-Anosov, meaning $w$ is full-support for $G$.

In terms of $\{a,b\}$ the word length of $w$ is $4P\cdot(2k(2n+1)+1)$.  Therefore one can let $l_2 = 4P\cdot(80\xi/c + 1)$.  For Theorem~\ref{srpa} one may take $$ K(S) = 2 \cdot |\Aut{H_1(S,\Z/3\Z)}| \cdot (320P\xi/c + 4P)^{2\xi}.$$\QED

\subsection{Finding sufficiently different reducibles}\label{gencon}
The purpose of this section is to prove Proposition~\ref{conj}.  In the hypothesis of the proposition, $a$ and $b$ are pure reducible elements of $\Mod{S'}$, where $S'$ is a domain in $S$, and $\grp{a,b}$ is irreducible.  Given any $k \geq 20/c(S)$ and $n \geq \xi(S) - 1$, let $u = (b^ka^k)^nb^k$ and $v = (a^kb^k)^na^k$.  To prove the proposition, we show that any choice of $\alpha \in \sigma(uau^{-1})$ and $\beta \in \sigma(vbv^{-1})$ together fill $S'$.  Necessarily $\alpha$ and $\beta$ will bound some domains $A \in \A{uau^{-1}}$ and $B \in \A{vbv^{-1}}$, and $A$ and $B$ together fill $S'$.

Without loss of generality we may assume $S' = S$.  

We prove Proposition~\ref{conj} in two steps.  First, we use Behrstock's Lemma (\ref{behrstock}) to follow the image of any $\alpha_1 \in \sigma(a)$ and $\beta_1 \in \sigma(b)$ under the composition of alternations of high powers of $a$ and $b$.  We find that $u(\alpha_1)$ and $v(\beta_1)$ have a certain subsurface projection property.  The second step is to prove that this property implies the two curves fill $S$.

For the purpose of step one, we introduce the \emph{overlap graph} for a pair of pure mapping classes.  Suppose $a$ and $b$ are pure nontrivial mapping classes generating the group $G$.  Let $\{A_i\}_{1\leq i \leq m}$ and $\{B_j\}_{1\leq j \leq n}$ be the domains of $a$ and $b$ respectively.  The overlap graph $\Ov{a,b}$ consists of a vertex for every domain of $a$ that overlaps with some domain of $b$, and for every domain of $b$ that overlaps with some domain of $a$---recall that two domains overlap if the boundary of each domain projects to the other.  Edges connect domains that overlap, so that one may color the vertices depending on whether they represent domains of $a$ or $b$, obtaining a bipartite graph.  Assign each edge length one so that $\Ov{a,b}$ has the usual path metric.

\begin{figure}[h]
\begin{center}
\includegraphics{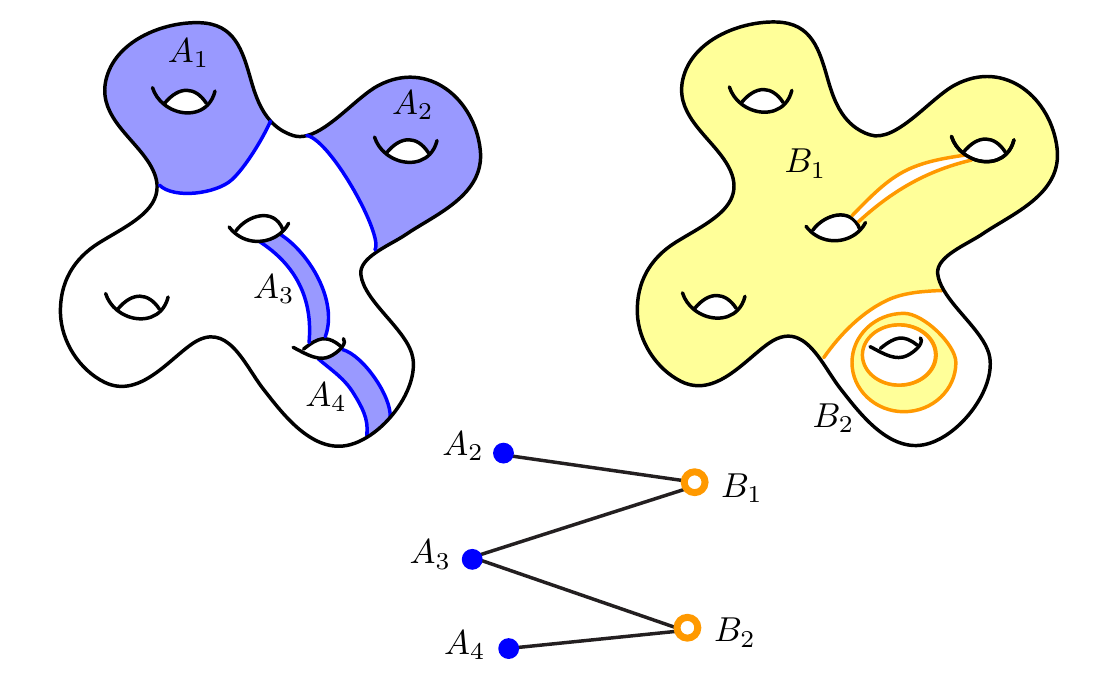}
\caption{\emph{Example of an overlap graph.}}
\end{center}
\end{figure}

Note that the foregoing definition works whether or not $G$ contains a pseudo-Anosov.  Now let us assume that, as in the hypotheses for Proposition~\ref{conj}, $G$ does contain a pseudo-Anosov, or equivalently, $\A{G}$ is the connected surface $S'$.  In particular $a$ and $b$ fix no common curve.  If $a$ is not pseudo-Anosov, then any domain $A_i$ has essential boundary consisting of curves not fixed by $b$.  This boundary must project to some domain $B_j$ of $b$.  If $B_j$ has no essential boundary, then it is all of $S'$ and $b$ is pseudo-Anosov.  Otherwise, the boundary projects to some domain of $a$.  If $\del B_j$ projects to $A_i$, then $B_j$ and $A_i$ overlap.  If not, then $A_i$ nests in $B_j$ by Lemma~\ref{threecases}, and $\del B_j$ projects to some other domain $A_k$.  In this case, $\del A_k$ in turn projects to $B_j$, because otherwise $B_j$ nests in $A_k$, implying $A_i$ nests in $A_k$, a contradiction.  So $B_j$ and $A_k$ overlap.  In either case $\Ov{a,b}$ has at least two vertices and an edge connecting them.  These observations imply

\begin{lmm}\label{empty}If $\A{G}$ is connected and $\Ov{a,b}$ is empty, then at least one of $a$ and $b$ is pseudo-Anosov.\end{lmm}

In the case where $\Ov{a,b}$ is not empty, we distinguish between the domains represented by vertices---call these domains \emph{overlappers}---and those not represented.  Revisiting the discussion above, we can extract a fact so handy we should name it:

\begin{lmm}\label{oerlap}When neither $a$ nor $b$ are pseudo-Anosov, any non-overlapper of one nests in an overlapper of the other.\end{lmm}

That observation helps us to the result below.

\begin{lmm}\label{connected}If $\A{G}$ is connected, then $\Ov{a,b}$ is connected.\end{lmm}

\begin{proof}Assuming $\Ov{a,b}$ is not empty, we show how to find a path between any two vertices.  Realize the domains as closed submanifolds with minimal pairwise intersection (i.e., ensure that boundary curves intersect essentially and use pairwise disjoint submanifolds to represent domains for the same mapping class).  For any two vertices $X$ and $Y$ of $\Ov{a,b}$, choose points $x$ and $y$ in the corresponding domains, and connect these by a path $p$.  Letting $A = \bigcup_i A_i$ and $B = \bigcup_j B_j$, one sees that each component of $S'\backslash(A \cup B)$ is a disk or boundary-parallel annulus with a boundary component consisting of pieces of $\del A \cup \del B$.  Thus one can isotope $p$ to lie entirely within $A \cup B$, and furthermore transversal to the boundary curves.  The path also lies entirely within overlappers, by Lemma~\ref{oerlap}.  Tracing the path gives a finite sequence of domains, each overlapping with the neighbor before and after, alternating between $A_i$'s and $B_j$'s by design.  The same sequence appears as a path in $\Ov{a,b}$ connecting $X$ and $Y$.\end{proof}

Now we can state precisely the first step to proving Proposition~\ref{conj}.  Assume $a$ and $b$ are as in Proposition~\ref{conj}, with the same definitions as above for $A = \bigcup_i A_i$ and $B = \bigcup_j B_j$.  Recall $c = c(S)$ is the minimal translation constant from Theorem~\ref{strans}, and $\xi(S)$ is surface complexity.

\begin{lmm}\label{tangle}Suppose $\alpha_1$ and $\beta_1$ are curves in $\sigma(a)$ and $\sigma(b)$ respectively, and $k$ and $n$ are integers such that $k \geq 20/c(S)$, and $n \geq \xi(S) - 1$.  For $\alpha = (b^ka^k)^nb^k(\alpha_1)$ and $\beta = (a^kb^k)^na^k(\beta_1)$, the following hold for any choice of $A_i$ or $B_j$.
\begin{itemize}
\item[(i)] If $\del A$ projects to $B_j$, then so does $\alpha$, and $\dis{B_j}{\alpha,\del A} \geq 14$.
\item[(ii)] If $\del B$ projects to $A_i$, then so does $\beta$, and $\dis{A_i}{\beta,\del B} \geq 14$.
\end{itemize}
\end{lmm}

\begin{proof}By symmetry we need only prove (i), the case for $\alpha$.  We can use the overlap graph to track the image of $\alpha_1$ under alternating applications of $b^k$ and $a^k$.  For ease of exposition, we will not always distinguish between overlapper domains and their representative vertices in $\Ov{a,b}$.

Let $Y_0$ be the vertices of $\Ov{a,b}$ corresponding to the overlappers of $b$ that intersect $\alpha_1$; note that $Y_0$ cannot be empty.  Let $X_0$ be the vertices adjacent to $Y_0$.  For $i \in \N$, let $Y_i$ be the vertices adjacent to $X_{i-1}$, and let $X_i$ be the vertices adjacent to $Y_i$.  In other words, $Y_i$ corresponds to $b$-domains that overlap with $a$-domains of $X_{i-1}$, and $X_i$ are the $a$-domains that overlap with $b$-domains of $Y_i$.

Observe that the vertices of $Y_{i+1}$ lie within a radius-two neighborhood of $Y_i$.  Because $\xi(S)$, the maximum number of disjoint curves on $S$, gives an upper bound on the number of domains for a mapping class in $\Mod{S}$, twice $\xi(S)$ bounds the number of vertices in $\Ov{a,b}$.  Because $\Ov{a,b}$ is connected, one sees that for $n \geq \xi(S)-1$, $Y_n$ consists of all the $b$-vertices, corresponding to all the overlapper domains of $b$.  If a domain $B_j$ of $b$ is not an overlapper, then by Lemma~\ref{oerlap} it nests in some domain of $a$, which precludes projection of $\del A$ to $B_j$.  To prove Lemma~\ref{tangle} it suffices to establish the following claim:

\begin{itemize}
\item[($\ast$)] For any domain $Y$ of $Y_n$, $\dis{Y}{(b^ka^k)^nb^k(\alpha_1),\del A} \geq 14$
\end{itemize}

We induct on $n$.  By definition, $\alpha_1$ projects to any domain $Y$ of $Y_0$.  Applying Theorem~\ref{strans},
$$\dis{Y}{b^k(\alpha_1),\alpha_1} \geq c|k| \geq 20 \geq 14$$
Because $\alpha_1$ is a component of the multicurve $\del A$, and because projection distances are diameters, we obtain the $n=0$ case of ($\ast$):
$$\dis{Y}{b^k(\alpha_1),\del A} \geq 14$$

Supposing ($\ast$) true for $n=m$, we prove it for $n=m+1$.  Consider any domain $Y$ of $Y_m$, and recall that $\pdiam{Y}{\del A} \leq 2$, by Lemma~\ref{projbound}.  For any $A_i$ overlapping $Y$, the triangle inequality gives
$$\dis{Y}{(b^ka^k)^mb^k(\alpha_1),\del A_i} \geq 12$$
In particular, $(b^ka^k)^mb^k(\alpha_1)$ projects to every $A_i$ that overlaps some $Y$ of $Y_m$; those $A_i$ are exactly the vertices of $X_m$.  In other words, for any $X \in X_m$, there exists $Y \in Y_m$ such that
$$\dis{Y}{(b^ka^k)^mb^k(\alpha_1),\del X} \geq 12$$
We may apply Behrstock's lemma, giving
$$\dis{X}{(b^ka^k)^mb^k(\alpha_1),\del Y} \leq 4$$
On the other hand, Theorem~\ref{strans} guarantees
$$\dis{X}{a^k(b^ka^k)^mb^k(\alpha_1),(b^ka^k)^mb^k(\alpha_1)} \geq c|k| \geq 20$$
Employing the triangle inequality,
$$\dis{X}{a^k(b^ka^k)^mb^k(\alpha_1),\del Y} \geq 16$$
The above holds for any $Y \in Y_m$, and $X \in X_m$.  Given any $Y' \in Y_{m+1}$, recall that $Y'$ corresponds to a domain overlapping some $X \in X_m$, which in turn overlaps with some $Y \in Y_m$.  We also know $\pdiam{X}{\del Y \cup \del Y'} \leq 2$, because $\del Y'$ and $\del Y$ are disjoint.  Another triangle inequality gives
$$\dis{X}{a^k(b^ka^k)^mb^k(\alpha_1),\del Y'} \geq 14$$
Again we can apply Behrstock's lemma.  In fact we can mirror the last four inequalites:
$$\dis{Y'}{a^k(b^ka^k)^mb^k(\alpha_1),\del X} \leq 4$$
$$\dis{Y'}{b^ka^k(b^ka^k)^mb^k(\alpha_1),a^k(b^ka^k)^mb^k(\alpha_1)} \geq 20$$
$$\dis{Y'}{b^ka^k(b^ka^k)^mb^k(\alpha_1),\del X} \geq 16$$
$$\dis{Y'}{b^ka^k(b^ka^k)^mb^k(\alpha_1),\del X'} \geq 14$$
where $X'$ is any domain of $X_{m+1}$.  Note that $\del X'$ is a subset of $\del A$, and recall $Y'$ was an arbitrary vertex of $Y_{m+1}$.  Thus we may re-write the last inequality to give the $m+1$ step of the induction claim: for any $Y \in Y_{m+1}$,
$$\dis{Y}{(b^ka^k)^{m+1}b^k(\alpha_1),\del A} \geq 14$$\end{proof}

We are ready for the second step to proving Proposition~\ref{conj}.  Its own proof will probably feel like d\'{e}j\`{a} vu.  Again, $a$ and $b$ are as in the proposition, with the same definitions as above for $A = \bigcup_i A_i$ and $B = \bigcup_j B_j$.  Note that in the following lemma, $\alpha$ and $\beta$ are arbitrary.

\begin{lmm}\label{fillcondition} Suppose $\alpha, \beta \in \V{S'}$ satisfy the following:
\begin{itemize}
\item[(i)] If $\del A$ projects to $B_j$, then so does $\alpha$, and $\dis{B_j}{\alpha,\del A} \geq 14$.
\item[(ii)] If $\del B$ projects to $A_i$, then so does $\beta$, and $\dis{A_i}{\beta,\del B} \geq 14$.
\end{itemize}
Then $\alpha$ and $\beta$ fill $S'$.\end{lmm}

\begin{proof}Given an arbitrary curve $\gamma \in \C{S'}$, we show it intersects either $\alpha$ or $\beta$.  Because $\gamma$ cannot be fixed by both $a$ and $b$, it projects to some domain, and we can choose this domain to be an overlapper (Lemma~\ref{oerlap}).  Without loss of generality we may assume $\gamma$ projects to a domain $Y$ of $b$, which overlaps with a domain $X$ of $a$.  Because multicurves have diameter-two projections (Lemma~\ref{projbound}), properties (i) and (ii) imply:
$$\dis{Y}{\alpha,\del X} \geq 12$$
$$\dis{X}{\beta,\del Y} \geq 12$$

Assuming $\gamma$ does not intersect $\alpha$, we show it must intersect $\beta$.  Again Lemma~\ref{projbound} implies $\dis{Y}{\alpha,\gamma} \leq 2$.  The following inequalities employ the triangle inequality, Behrstock's lemma, and another triangle inequality.
$$\dis{Y}{\gamma,\del X} \geq 10$$
$$\dis{X}{\gamma,\del Y} \leq 4$$
$$\dis{X}{\gamma,\beta} \geq 8$$
In yet another instance of Lemma~\ref{projbound}, the last inequality shows $\gamma$ and $\beta$ must intersect.\end{proof}
\medskip
\noindent \emph{Proof of Proposition~\ref{conj}.}  Let $u = (b^ka^k)^na^k$ and $v = (a^kb^k)^nb^k$.  Taken together, Lemmas~\ref{tangle} and \ref{fillcondition} show that, for any $\alpha \in \sigma(uau^{-1})$ and $\beta \in \sigma(vbv^{-1})$, $\alpha$ and $\beta$ fill $S'$.  Therefore $a_1 = uau^{-1}$ and $b_1 = vbv^{-1}$ are sufficiently different pure reducible mapping classes.\QED

\subsection{Constructions}\label{blueprint}
Here we give the pseudo-Anosov construction at the heart of the proof of our Main Theorem.  That is, we prove Proposition~\ref{zigzag} of Section \ref{prf}, which states that, in a group generated by two sufficiently different pure reducible mapping classes, every element is pseudo-Anosov except those conjugate to powers of the generators.  The second part of this section upgrades Proposition~\ref{zigzag} to Proposition~\ref{coreconstruction} given in the introduction, by proving the all-pseudo-Anosov subgroups are convex cocompact.

\subsubsection{Of pseudo-Anosovs}

Recall that a pseudo-Anosov is characterized by having infinite-diameter orbits in the curve complex.  Thus we will know $w$ is a pseudo-Anosov if distances $\Dis{\gamma,w^n(\gamma)}$ grow as $n$ increases, where $\gamma$ is some vertex in $\C{S}$.  Proposition~\ref{zigzag} is an application of Lemma~\ref{geo} below, which uses the Bounded Geodesic Image Theorem (Theorem~\ref{bgi}) to glean geometric information about sequences of geodesics.  The author learned this ``bootstrapping'' strategy from Chris Leininger, who takes a similar tack in \cite{Le}.  In the same vein one also has Proposition 5.2 in \cite{KL1}, but here the domains are not restricted.

Let $M = M(S)$ be the constant from Theorem~\ref{bgi}.  Below, properties (ii) and (iii) ensure that the projection in (iv) is well-defined.

\begin{lmm}\label{geo}Suppose $\{Y_j\}$ is a sequence of domains in $S$, and $\{X_j\}$ a sequence of subsets of $\V{S}$, for which the following properties hold for all $j$:
\begin{itemize}
\item[(i)] $\diam{X_j} \leq 2$
\item[(ii)] $X_j$ and $X_{j+1}$ are pairwise disjoint
\item[(iii)] The set of curves that project trivially to $Y_j$ is a subset of $X_j$
\item[(iv)] $\dis{Y_j}{w_{j-1},w_{j+1}} > 2M$ for any choice of $w_{j-1} \in X_{j-1}$, $w_{j+1} \in X_{j+1}$
\end{itemize}
Then for any $w_i \in X_i$ and $w_{i+k} \in X_{i+k}$, the geodesic $[w_i,w_{i+k}]$ contains a vertex from $X_j$ for $i \leq j \leq i+k$.   Also, the $X_j$ are pairwise disjoint.  In particular, $[w_i,w_{i+k}]$ has length at least $k$.\end{lmm}

\begin{proof}First let us induct on $k$ the claim that, for any $w_i \in X_i$ and $w_{i+k} \in X_{i+k}$, the geodesic $[w_i,w_{i+k}]$ contains a vertex from $X_j$ for $i \leq j \leq i+k$.  This is vacuously true for $k = 1$.

The general induction step works even for $k=2$, but let us separate this case to highlight its use of the Bounded Geodesic Image Theorem, Theorem~\ref{bgi}.  By property (iv), $\dis{Y_{i+1}}{w_{i},w_{i+2}} > 2M > M$, so Theorem~\ref{bgi} requires $[w_i,w_{i+2}]$ contain a vertex $v$ disjoint from $Y_{i+1}$.  By property (iii), $v$ is contained in $X_{i+1}$.  Of course, the endpoints of $[w_i,w_{i+2}]$ lie in $X_i$ and $X_{i+2}$, so the induction claim holds.

Now assume $k > 2$ and for any $w_i \in X_i, w_{i+k-1} \in X_{i+k-1}$, the geodesic $[w_i,w_{i+k-1}]$ intersects $X_j$ for $i \leq j \leq i+k-1$.  The main task is to show $[w_i,w_{i+k}]$ intersects $X_{i+k-1}$.

Choosing any $w_{i+k-2} \in X_{i+k-2}$, our first step is to show one can pick a geodesic $[w_i,w_{i+k-2}]$ avoiding $X_{i+k-1}$.  Suppose we are given a geodesic that doesn't, that is, for some $w'_{i+k-1} \in X_{i+k-1}$,
\[ [w_i,w_{i+k-2}] = [w_i,w'_{i+k-1}] \cup [w'_{i+k-1},w_{i+k-2}]  .\]
Let us require $w'_{i+k-1}$ to be the first vertex along $[w_i,w_{i+k-2}]$ belonging to $X_{i+k-1}$, so that only the last vertex of $[w_i,w'_{i+k-1}]$ lies in $X_{i+k-1}$.  The induction hypothesis applied to the segment $[w_i,w'_{i+k-1}]$ implies that, for some $w'_{i+k-2} \in X_{i+k-2}$,
\[ [w_i,w_{i+k-2}] = [w_i,w'_{i+k-2}] \cup [w'_{i+k-2},w'_{i+k-1}] \cup [w'_{i+k-1},w_{i+k-2}]  .\]
Property (ii) ensures the second and third segments above each have length at least one.  Thus their union $[w'_{i+k-2},w_{i+k-2}]$ has length at least two.  Property (i) tells us we can replace it with a length-2 geodesic contained entirely in $X_{i+k-2}$, giving a new $[w_i,w_{i+k-2}]$ avoiding $X_{i+k-1}$.

Property (iii) ensures that every vertex of $[w_i,w_{i+k-2}]$ projects nontrivially to $Y_{i+k-1}$, as does $w_{i+k}$.  Projection distances make sense and Theorem~\ref{bgi}, the Bounded Geodesic Image Theorem, applies:
\[\dis{Y_{i+k-1}}{w_i,w_{i+k-2}} \leq \pdiam{Y_{i+k-1}}{[w_i,w_{i+k-2}]} \leq M  .\]
Finally, a triangle inequality:
\begin{eqnarray*}
\pdiam{Y_{i+k-1}}{[w_i,w_{i+k}]} & \geq & \dis{Y_{i+k-1}}{w_i,w_{i+k}}\\
& \geq & \dis{Y_{i+k-1}}{w_{i+k-2},w_{i+k}} - \dis{Y_{i+k-1}}{w_i,w_{i+k-2}} \\
& > & 2M - M = M
\end{eqnarray*}

Again by Theorem~\ref{bgi}, we know $[w_i,w_{i+k}]$ intersects $X_{i+k-1}$.  The end of the induction is easy: for some $w_{i+k-1}$ in $X_{i+k-1}$,
\[ [w_i,w_{i+k}] = [w_i,w_{i+k-1}] \cup [w_{i+k-1},w_{i+k}]  .\]
The induction hypothesis says the first segment on the right intersects each $X_j$, $i \leq j \leq i+k-1$.  So the geodesic on the left intersects $X_j$, $i \leq j \leq i+k$, as required.

Finally we check that the sets $X_j$ are pairwise disjoint.  Suppose $z \in X_i \cap X_{i+k}$ for some nonzero $k$.  By the part of the lemma already proved, the geodesic $[z,z]$ contains vertices in $X_j$ for $i \leq j \leq i+k$: simply put, all those $X_j$ intersect at $z$.  But consecutive $X_j$ do not intersect.\end{proof}

\noindent \textbf{Proof of Proposition~\ref{zigzag}.}  Recall that $Q = Q(S) = \max\{3,(2M + 4)/c\}$ where $M=M(S)$, the constant from Theorem~\ref{bgi}, and $c = c(S)$, the constant from Corollary \ref{strans}.  The hypothesis states that $a$ and $b$ are pure reducible mapping classes, with $A \in \A{a}$ and $B \in \A{b}$, and $A$ and $B$ together fill $S$.  We will show that for any $m,n > Q$, $\grp{a^m, b^n}$ is free and nontrivial elements of $\grp{a^m, b^n}$ are either pseudo-Anosov or conjugate to powers of $a^m$ or $b^n$.

If $S$ is a torus or four-punctured sphere, the only pure reducible mapping classes are Dehn twists about curves, and any pair of curves fill the surface.  In this case the result for $Q = 3$ is due to \cite{HT}; see also \cite{Th, Pe, Ish}.  For the rest of the proof let us assume $\xi(S) > 1$.

Let $n,m \geq Q$ be arbitrary integers.  Let $\alpha$ be a component of $\del A$ and $\beta$ of $\del B$.  Let $C_A \subset \V{S}$ be the vertices with empty projection to $A$, and define $C_B$ analogously.  Observe that $C_A$ and $C_B$ sit in $1$-neighborhoods of $\alpha$ and $\beta$, respectively, so they each have diameter 2 in $\C{S}$.  Because any curve intersects either $A$ or $B$, $C_A$ and $C_B$ contain no common vertices---this is the only place we use the fact that $A$ and $B$ fill $S$.

Here's the whole point of our choice of $Q$: for any nonzero integer $k$,
$$\dis{A}{C_B,a^{mk}(C_B)} \geq \dis{A}{\beta, a^{mk}(\beta)} \geq |mk|c > Qc \geq 2M + 4$$
$$\dis{B}{C_A,b^{nk}(C_A)} \geq \dis{B}{\alpha, b^{nk}(\alpha)} \geq |nk|c > Qc \geq 2M + 4$$
Notice that the domains $A, B$ and sets $C_A, C_B$ play the same roles for $a^m$ and $b^n$ as they do for $a$ and $b$.  To ease the burden of excessive exponents, we replace $a^m$ with $a$ and $b^n$ with $b$ for the remainder of the proof.  In this new notation, our goal is to show that $\grp{a,b}$ is free, and every nontrivial element of $\grp{a,b}$ is pseudo-Anosov except those conjugate to powers of $a$ or $b$.  We have already established that, for any nonzero $k$,
$$\dis{A}{C_B,a^k(C_B)} > 2M + 4 \qquad \text{and} \qquad \dis{B}{C_A,b^k(C_A)} > 2M + 4$$
In particular, for any $\gamma,\gamma' \in C_B$ or $\delta,\delta' \in C_A$,
\begin{equation}\label{eq:bgi}\dis{A}{\gamma,a^k(\gamma')} > 2M \qquad \text{and} \qquad \dis{B}{\delta,b^k(\delta')} > 2M\end{equation}
We use this shortly to apply Lemma~\ref{geo}.

In the abstract free group on $a$ and $b$, a word $w$ has reduced form $w = s_{1}s_{2}\dots s_{R}$, where the \emph{syllables} $s_i$ are nontrivial powers of either $a$ or $b$, and $s_i$ is a power of $a$ if and only if $s_{i\pm1}$ is a power of $b$ (i.e.\ powers of $a$ and $b$ alternate).  Define \emph{powerblind} word length $|\cdot|_*$ by $|w|_* = R$, the number of syllables of $w$.

\textbf{Claim.} For any word $w$, either $\Dis{w(\alpha),\alpha} \geq |w|_*$ or $\Dis{w(\beta),\beta} \geq |w|_*$.

It immediately follows that $\grp{a,b}$ is a rank-two free group.  If $|w|_*$ is even, it is easy to check $|w^n|_* = n|w|_*$.  The claim implies $\Dis{w^n(\gamma),\gamma} \geq n|w|_*$, where $\gamma$ is $\alpha$ or $\beta$, depending on $w$. In particular the orbit of $w$ has infinite diameter in the curve complex, so $w$ is a pseudo-Anosov.  If $|w|_*$ is odd and neither conjugate to a power of $a$ nor of $b$, then it is conjugate to $v$ such that $|v|_*$ is even.  As $v$ is pseudo-Anosov, so is its conjugate $w$.

It remains to prove the claim.  Towards this, we describe a sequence of domains and $\V{S}$-subsets fulfilling the hypotheses of Lemma~\ref{geo}.  Let $w_0$ be the identity, $w_1 = s_1$, $w_2 = s_1s_2$, and so forth, so that $w_i$ is the word formed by the first $i$ syllables of $w$.  If $s_1$ is a power of $a$, let $I(a)$ correspond to the even integers, and $I(b)$ to the odds; if $s_1$ is a power of $b$, switch the roles of even and odd.  Define $R$-length sequences of vertices $\gamma_r$, domains $Y_r$, and sets $X_r$ as follows.
\begin{align*}
\gamma_r &= w_r(\alpha) \quad &Y_r &= w_r(A) \quad &X_r &= w_r(C_A) & \forall r \in I(a)\cap[0,R] &\\
\gamma_r &= w_r(\beta) \quad &Y_r &= w_r(B) \quad &X_r &= w_r(C_B) &\forall r \in I(b)\cap[0,R] &
\end{align*}
In addition, set $\{\gamma_{-1}, Y_{-1}, X_{-1}\}$ equal to $\{\beta,B,C_B\}$, if $s_1$ is a power of $a$, or $\{\alpha,A,C_A\}$, if $s_1$ is a power of $b$.

Each $X_j$ is isometric to $C_A$ or $C_B$, and furthermore pairs $X_j,X_{j+1}$ are isometric to the pair $C_A,C_B$, disregarding order.  Therefore the sequence $\{Y_j,X_j\}$ meets conditions (i) -- (iii) of Lemma~\ref{geo}.  Condition (iv) requires $\dis{Y_j}{v_{j-1},v_{j+1}} > 2M$ for any choice of $v_{j-1} \in X_{j-1}$, $v_{j+1} \in X_{j+1}$.  For $j=0$, this condition simply restates one of the inequalities in \eqref{eq:bgi} above, so let us suppose $j \geq 1$.  Without loss of generality, assume $Y_j = w_j(A)$ and $v_{j-1} = w_{j-1}(\gamma_{-}),v_{j+1} = w_{j+1}(\gamma_{+})$ for some $\gamma_{-}, \gamma_{+} \in C_B$.  Because subsurface projection commutes naturally with the action of the mapping class group,
$$\dis{w_j(A)}{v_{j-1},v_{j+1}} = \dis{A}{w_j^{-1}(v_{j-1}),w_j^{-1}(v_{j+1})} = \dis{A}{s_j^{-1}(\gamma_{-}),s_{j+1}(\gamma_{+})}$$
Exactly one of $s_j$ and $s_{j+1}$ is a power of $a$, while the other is a power of $b$.  In any case one knows that, for some $\gamma,\gamma' \in C_B$ and nonzero $k$,
$$\dis{A}{s_j^{-1}(\gamma_{-}),s_{j+1}(\gamma_{+})} = \dis{A}{a^k(\gamma),\gamma'}$$
We used the fact that $b$ fixes $C_B$ setwise.  Applying inequality \eqref{eq:bgi} above, we can conclude $\dis{w_j(A)}{v_{j-1},v_{j+1}} > 2M$.

\begin{figure}[h]
\begin{center}
\includegraphics{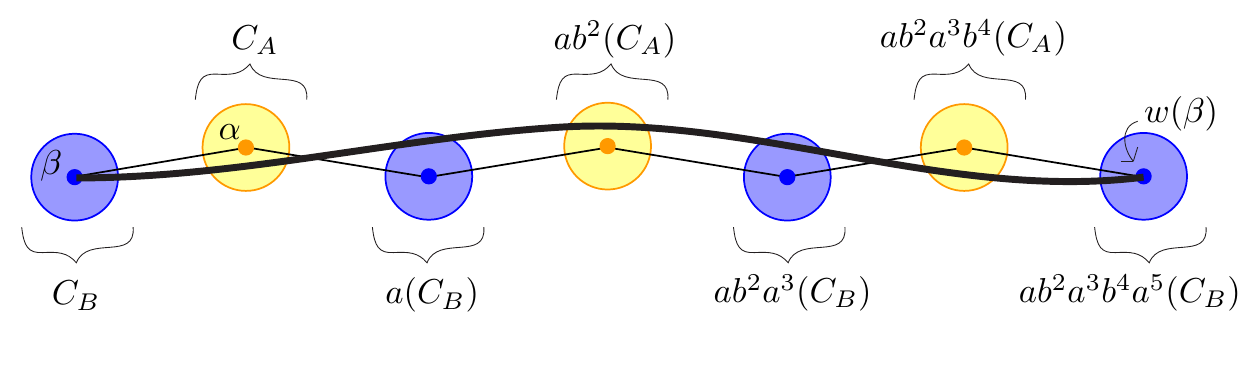}
\caption{\emph{The sequence defined in the proof of Proposition~\ref{zigzag}.}  Here $w=ab^2a^3b^4a^5$.  Vertices correspond to $\gamma_i$ and shaded circles represent $X_i$, which each sit in the 1-neighborhood of $\gamma_i$.  The heavy, smooth line is a geodesic connecting $\beta$ and $w(\beta)$; it intersects each of the $X_i$.}
\end{center}
\end{figure}

By Lemma~\ref{geo}, the geodesics $[\gamma_{-1},\gamma_R]$ and $[\gamma_0,\gamma_R]$ have lengths at least $R+1$ and $R$ respectively.  Depending on $w$, one of these geodesics is either $[\alpha,w(\alpha)]$ or $[\beta,w(\beta)]$.  This proves the claim, thus the lemma.\QED

\noindent \textbf{Remark.}  One can ask whether Proposition~\ref{zigzag} above can be proven for $n$-tuples of pure mapping classes subject to appropriate conditions.  For example, if pairwise they fulfill the requirements for $a$ and $b$ in the lemma, do sufficiently high powers generate rank $n$ free groups?  One would need to avoid the trivial counterexamples arising from redundant generating sets such as $\{a,b,a^kba^{-k}\}$.  Excluding this latter possibility, the immediate adaptation of the proof of Lemma~\ref{zigzag} for $n$-tuples only gives a $Q$ that depends on the particular $n$-tuple; proving that some $Q$ works for any $n$-tuple seems to require more maneuvering.

\subsubsection{Of convex cocompact all-pseudo-Anosov subgroups}\label{schottky}
Now we upgrade Proposition~\ref{zigzag} to Proposition~\ref{coreconstruction} of the introduction.  The action of a group $G < \Mod{S}$ on the curve complex gives a \emph{quasi-isometric embedding} $G \into \C{S}$ if, for some $\gamma \in \V{S}$, $K > 1$, and $C \geq 0$, and for all $w \in G$,\begin{equation}\label{eq:qi}K |w| + C \geq \Dis{\gamma,w(\gamma)} \geq |w|/K - C\end{equation} where $|w|$ gives word length with respect to some metric on $G$.  From \eqref{eq:qi} one sees that such a group contains no non-trivial reducible elements.  Let us call a group \emph{convex cocompact} if it fulfills \eqref{eq:qi} for some $\gamma$, $K$, and $C$ as above.  Hamenst\"{a}dt and Kent-Leininger proved that this is equivalent to the definition of convex cocompact mapping class subgroups introduced by Farb and Mosher in \cite{FM}.  Convex cocompact subgroups which are free of finite rank are called Schottky.  Farb and Mosher proved these are common in mapping class groups of closed surfaces, and Kent and Leininger include a new proof which also works for non-closed surfaces:

\begin{thm}[Abundance of Schottky groups \cite{FM}, \cite{KL2}] Given a finite set of pseudo-Anosovs $\{g_1, g_2, ... g_k\}$ which are independent (i.e., no pairs of powers commute), there exists $l$ such that for all $m > l$,  $\{g^m_1, g^m_2, ... g^m_k\}$ is Schottky.\end{thm}

Fujiwara found a uniform bound for the above theorem in the case of two-generator subgroups:

\begin{thm}[Fujiwara \cite{Fu2}, full version]  There exists a constant $L = L(S)$ with the following property.  Suppose $a,b \in \Mod{S}$ are independent pseudo-Anosov elements.  Then for any $n,m \geq L$, $\grp{a^n,b^m}$ is Schottky.\end{thm}

Proposition~\ref{coreconstruction} provides a source of Schottky subgroups with arbitrary rank.  It simply adds to Proposition~\ref{zigzag} the claim that finitely generated all-pseudo-Anosov subgroups of $\grp{a,b}$ are Schottky.  Recall that $a$ and $b$ are pure mapping classes with essential reduction curves $\alpha$ and $\beta$ respectively, such that $\alpha \cup \beta$ fills $S$.  Proposition~\ref{zigzag} tells us that $G =\grp{a,b}$ is a free group, and its all-pseudo-Anosov subgroups are exactly those containing no conjugates of powers of $a$ or $b$.  For example, any subgroup of the commutator group $[G,G]$ qualifies.

\medskip \noindent \emph{Proof of Proposition~\ref{coreconstruction}.}  Let $G = \grp{a,b}$ as in the proposition, and let $H$ be a finitely generated, all-pseudo-Anosov subgroup of $G$.  To show that $H$ is Schottky, we need quasi-isometry constants for the inequalities in \eqref{eq:qi}.  For an orbit embedding of a finitely generated group, the upper bound of \eqref{eq:qi} always comes for free, using any $K$ greater than the largest distance some generator translates $\gamma$.  So our work is the lower bound.  In the proof of Lemma~\ref{zigzag} we saw that $\Dis{\gamma,w(\gamma)} \geq |w|_*$ where $\gamma$ is one of $\alpha$ or $\beta$, depending on $w$.  It is not hard to check that, in general, $\Dis{\alpha,w(\alpha)} \geq |w|_* - 1$ (the same is true replacing $\alpha$ with $\beta$).  However, powerblind word length is not a word metric with respect to any finite generating set for $H$.  In what follows, we define a convenient generating set for $H$ such that, letting $|\cdot|_H$ denote the corresponding word metric, we have $K|w|_* \geq |w|_H$, where $K$ is the size of this generating set.  Then we can conclude that $\Dis{\alpha,w(\alpha)} \geq |w|_H/K - 1$, completing the proof.

The existence of this convenient generating set has no relation to our setting of subgroups of mapping class groups, so we isolate this fact as a separate technical lemma.  We only need that $G$ is a rank two free group generated by $a$ and $b$.  Suppose $H$ is a finitely generated subgroup.  For $w \in H$, let $|w|$ denote its length in $G$ with respect to the generating set $\{a,b\}$.  Let $H_l = \{h \in H \colon |h| \leq l\}$ and choose $L$ so that $H_L$ generates $H$.  Define $|\cdot|_H$ as word length in $H$ with respect to $H_L$.  Let $K$ be the number of elements in $H_L$.

\begin{lmm}\label{technical}If $H$ contains no element conjugate to a power of $a$ or $b$, then $K|w|_* \geq |w|_H$.
\end{lmm}

\begin{proof}Suppose $w \in H$ has length $|w|_H = n$, and $w = h_1h_2 \cdots h_n$, where $h_i \in H_L$.  One knows that $|h_ih_{i+1}| > |h_i|,|h_{i+1}|$ because otherwise $h_ih_{i+1} \in H_L$, which contradicts that $|w|_H = n$ (one could take a shorter path to $w$ in the Cayley graph of $H$ with respect to $H_L$, by replacing $h_i$ and $h_{i+1}$ with their product, another generator).  In particular this means that strictly less than half of the word $h_i$, written as a product of $a$'s and $b$'s, gets canceled by a piece of the word $h_{i+1}$ in $a$'s and $b$'s.  Likewise, strictly less than half gets canceled by a piece of the word $h_{i-1}$.  Therefore, at least the middle letter of $h_i$, if $|h_i|$ is odd (the middle pair of letters if $|h_i|$ is even) gets contributed to the $\{a,b\}$-spelling of the word $w$. Incidentally, this is showing that $|w| \geq |w|_H$, implying that finitely generated subgroups of the rank-two free group are quasi-isometrically embedded.

Call the middle letter or pair of letters of each $h_i$ its \emph{core}.  Powerblind word length $|w|_*$ can be shorter than $|w|_H$ only if a string of consecutive $h_i$'s, say, $h_ih_{i+1}\cdots h_{k}$, all have $a$ or $a^2$ at their core, or if they all have $b$ or $b^2$ at their core, and these cores contribute to the same syllable (power of a or b) in the $\{a,b\}$-spelling of $w$.  In that case one can write, for $i \leq j \leq k$, $$h_j = u_j \cdot x^{e(j)} \cdot v_j$$ where $x$ is $a$ or $b$ and $x^{e(j)}$ includes the core of each $h_j$.  Furthermore, $v_j = u_{j+1}^{-1}$, so that $$h_ih_{i+1}\cdots h_{k} = u_ix^Nv_{k}$$ where $N = \sum_{i \leq j \leq k} e(j)$.  It is possible $u_i$ or $v_k$ are empty words, but $N$ cannot be zero, because otherwise $|h_ih_{i+1}\cdots h_{k}| \leq |h_i|/2 + |h_k|/2 \leq L$, meaning the entire string can be replaced with a single element of $H_L$, contradicting the fact that $|w|_H = n$.  In this context, suppose $h_i = h_k$.  Then $u_k = u_ix^p$ for some $p$.  But because $v_{k-1}=u_k^{-1}$, one has $$h_i\cdots h_{k-1} = u_ix^{N'}u_i^{-1}$$ for $N' = N - e(k) - p$.  However, the stipulation that $H$ contains no elements of $G$ conjugate to powers of $a$ or $b$, precludes this scenario.  Thus if any consecutive string $h_i\cdots h_k$ in the $H_L$-spelling of $w$ contributes to the same syllable in the $\{a,b\}$-spelling of $w$, that string includes at most one instance of each element of $H_L$.

We have demonstrated a correspondence between letters $h_i$ and syllables of $w$ written with respect to $H_L$ and $\{a,b\}$ respectively: each letter corresponds to at least one syllable (the one in which its core appears), and at most $K$ letters correspond to the same syllable.  Therefore $K|w|_* \geq |w|_H$.\end{proof}

As described above, Lemma \ref{technical} completes the proof of Proposition \ref{coreconstruction}.\QED

\bibliography{biblio}

\providecommand{\bysame}{\leavevmode\hbox to3em{\hrulefill}\thinspace}
\providecommand{\MR}{\relax\ifhmode\unskip\space\fi MR }
\providecommand{\MRhref}[2]{%
  \href{http://www.ams.org/mathscinet-getitem?mr=#1}{#2}
}
\providecommand{\href}[2]{#2}
\begin{thebibliography}{BLM83}

\bibitem[Beh06]{Be}
Jason Behrstock, \emph{Asymptotic geometry of the mapping class group and
  {T}eichm\"uller space}, Geom. Topol. \textbf{10} (2006), 1523--1578.

\bibitem[BLM83]{BLM}
Joan~S. Birman, Alex Lubotzky, and John McCarthy, \emph{Abelian and solvable
  subgroups of the mapping class groups}, Duke Math. J. \textbf{50} (1983),
  no.~4, 1107--1120.

\bibitem[FM]{FM}
Benson Farb and Dan Margalit, \emph{A primer on mapping class groups}, in
  preparation, http://www.math.utah.edu/$\sim$margalit/primer/.

\bibitem[Fuj08]{Fu1}
Koji Fujiwara, \emph{Subgroups generated by two pseudo-{A}nosov elements in a
  mapping class group. {I}. {U}niform exponential growth}, Groups of
  diffeomorphisms, Adv. Stud. Pure Math., vol.~52, Math. Soc. Japan, Tokyo,
  2008, pp.~283--296.

\bibitem[Fuj09]{Fu2}
\bysame, \emph{Subgroups generated by two pseudo-{A}nosov elements in a mapping
  class group. {II}. {U}niform bound on exponents}, preprint (2009),
  arXiv:0908.0995v1.

\bibitem[Ham05]{H}
Ursula Hamenst{\"a}dt, \emph{Word hyperbolic extensions of surface groups},
  preprint (2005), arXiv:0807.4891v2.

\bibitem[Har81]{Ha}
W.~J. Harvey, \emph{Boundary structure of the modular group}, 245--251.

\bibitem[Hem01]{He}
John Hempel, \emph{3-manifolds as viewed from the curve complex}, Topology
  \textbf{40} (2001), no.~3, 631--657.

\bibitem[HM09]{HM}
Michael Handel and Lee Mosher, \emph{Subgroup classification in
  $\mathrm{Out}({F}_n)$}, preprint (2009), arXiv:0908.1255v1.

\bibitem[HT02]{HT}
Hessam Hamidi-Tehrani, \emph{Groups generated by positive multi-twists and the
  fake lantern problem}, Algebr. Geom. Topol. \textbf{2} (2002), 1155--1178
  (electronic).

\bibitem[Ish96]{Ish}
Atsushi Ishida, \emph{The structure of subgroup of mapping class groups
  generated by two {D}ehn twists}, Proc. Japan Acad. Ser. A Math. Sci.
  \textbf{72} (1996), no.~10, 240--241.

\bibitem[Iva88]{Iv1}
Nikolai~V. Ivanov, \emph{The rank of {T}eichm\"uller modular groups}, Mat.
  Zametki \textbf{44} (1988), no.~5, 636--644, 701.

\bibitem[Iva92]{Iv2}
\bysame, \emph{Subgroups of {T}eichm\"uller modular groups}, Translations of
  Mathematical Monographs, vol. 115, American Mathematical Society, Providence,
  RI, 1992, Translated from the Russian by E. J. F. Primrose and revised by the
  author.

\bibitem[KL08a]{KL2}
Richard~P. Kent, IV and Christopher~J. Leininger, \emph{Shadows of mapping
  class groups: capturing convex cocompactness}, Geom. Funct. Anal. \textbf{18}
  (2008), no.~4, 1270--1325.

\bibitem[KL08b]{KL1}
\bysame, \emph{Uniform convergence in the mapping class group}, Ergodic Theory
  Dynam. Systems \textbf{28} (2008), no.~4, 1177--1195.

\bibitem[Lei]{Le}
Chris Leininger, \emph{Graphs of {V}eech groups}, in preparation.

\bibitem[Man10]{Ma}
Johanna Mangahas, \emph{Uniform uniform exponentional growth of subgroups of
  the mapping class group}, Geom. Funct. Anal. \textbf{19} (2010), no.~5,
  1468--1480.

\bibitem[McC85]{Mc}
John McCarthy, \emph{A ``{T}its-alternative'' for subgroups of surface mapping
  class groups}, Trans. Amer. Math. Soc. \textbf{291} (1985), no.~2, 583--612.

\bibitem[Min03]{Mi}
Yair~N. Minsky, \emph{The classification of {K}leinian surface groups, {I}:
  {M}odels and bounds}, preprint (2003), arXiv:math/0302208v3.

\bibitem[MM99]{MM1}
Howard~A. Masur and Yair~N. Minsky, \emph{Geometry of the complex of curves.
  {I}. {H}yperbolicity}, Invent. Math. \textbf{138} (1999), no.~1, 103--149.

\bibitem[MM00]{MM2}
\bysame, \emph{Geometry of the complex of curves. {II}. {H}ierarchical
  structure}, Geom. Funct. Anal. \textbf{10} (2000), no.~4, 902--974.

\bibitem[Mos]{Mo}
Lee Mosher, personal communication, MSRI Course on Mapping Class Groups, Fall
  2007.

\bibitem[Pen88]{Pe}
Robert~C. Penner, \emph{A construction of pseudo-{A}nosov homeomorphisms},
  Trans. Amer. Math. Soc. \textbf{310} (1988), no.~1, 179--197.

\bibitem[Poe79]{FLP}
Fathi~Laudenbach Poenaru, \emph{Travaux de {T}hurston sur les surfaces},
  Ast\'erisque, vol.~66, Soci\'et\'e Math\'ematique de France, Paris, 1979,
  S{\'e}minaire Orsay, With an English summary.

\bibitem[SW92]{SW}
Peter~B. Shalen and Philip Wagreich, \emph{Growth rates, {$Z_p$}-homology, and
  volumes of hyperbolic {$3$}-manifolds}, Trans. Amer. Math. Soc. \textbf{331}
  (1992), no.~2, 895--917.

\bibitem[Thu88]{Th}
William~P. Thurston, \emph{On the geometry and dynamics of diffeomorphisms of
  surfaces}, Bull. Amer. Math. Soc. (N.S.) \textbf{19} (1988), no.~2, 417--431.

\end{thebibliography}
\end{document}